\documentclass[12pt, reqno]{amsart}
\usepackage{amssymb, amsthm, amsfonts, amsmath, amscd}
\usepackage{enumitem}
\usepackage[utf8]{inputenc}
\usepackage[margin=1in]{geometry}
\usepackage{graphicx}
\usepackage{mathdots}
\usepackage{tikz}
\usepackage{tikz-cd}
\usepackage{mathrsfs}
\usepackage{hyperref}
\usepackage{array}
\usepackage{multirow}
\usepackage{comment}
\usepackage[normalem]{ulem}

\theoremstyle{plain}

\newtheorem{theorem}{Theorem}[section]
\newtheorem{lemma}[theorem]{Lemma}
\newtheorem{proposition}[theorem]{Proposition}
\newtheorem{corollary}[theorem]{Corollary}

\theoremstyle{definition}
\newtheorem{definition}[theorem]{Definition}
\newtheorem{remark}[theorem]{Remark}
\newtheorem{example}[theorem]{Example}

\newtheorem{problem}[theorem]{Problem}

\newcommand{\Y}{\mathbb Y}
\newcommand{\C}{\mathbb C}
\newcommand{\R}{\mathbb R}

\newcommand{\Z}{\mathbb Z}
\newcommand{\F}{\mathbb F}
\newcommand{\FF}{\wt\F}

\newcommand{\Ga}{\Gamma}
\newcommand{\Om}{\Omega}
\newcommand{\Si}{\Sigma}
\newcommand{\al}{\alpha}
\newcommand{\be}{\beta}
\newcommand{\ga}{\gamma}
\newcommand\de{\delta}
\newcommand{\la}{\lambda}
\newcommand{\eps}{\epsilon}
\newcommand{\vp}{\varphi}
\newcommand{\si}{\sigma}
\renewcommand{\th}{\theta}
\newcommand{\om}{\omega}

\newcommand{\Ran}{\operatorname{Ran}}
\newcommand{\rk}{\operatorname{rk}}
\newcommand{\Sym}{\operatorname{Sym}}

\newcommand{\HL}{{\operatorname{HL}}}
\newcommand{\Nil}{\operatorname{Nil}}
\newcommand{\NSin}{\operatorname{NSin}}
\newcommand{\Rad}{\operatorname{Rad}}
\newcommand{\Mat}{\operatorname{Mat}}
\newcommand{\Lat}{\operatorname{Lat}}

\newcommand{\ex}{\operatorname{ex}}
\newcommand{\Herm}{\operatorname{Herm}}

\newcommand\Cyl{\operatorname{Cyl}}
\newcommand\Planch{{\operatorname{Planch}}}
\newcommand\Haar{{\operatorname{Haar}}}
\newcommand\colsp{\operatorname{colsp}}
\newcommand\rowsp{\operatorname{rowsp}}

\newcommand{\wt}{\widetilde}
\newcommand{\wh}{\widehat}

\newcommand{\M}{\mathcal P}
\newcommand{\B}{\mathcal B}

\newcommand{\g}{\mathfrak g}
\newcommand{\gl}{\mathfrak{gl}}
\newcommand{\uu}{\mathfrak u}
\newcommand{\La}{\mathbb L}
\newcommand{\e}{\mathscr E}
\renewcommand{\o}{\mathscr O}
\newcommand{\Li}{L}
\newcommand{\T}{\mathcal T}
\newcommand{\PGL}{\M^{GL(\infty, q)}}
\newcommand{\PUeven}{\M^{U(2\infty, q^2)}}
\newcommand{\PUodd}{\M^{U(2\infty+1, q^2)}}

\newcommand{\GLB}{\mathbb{GLB}}
\newcommand{\UB}{\mathbb{UB}}
\newcommand{\Tr}{\operatorname{Tr}}

\newcommand\ccdot{\,\cdot\,}
\newcommand\glinfty{{GL(\infty,q)}}
\newcommand\m{\boldsymbol{m}}
\newcommand\Harm{\mathcal{H}}

\numberwithin{equation}{section}

\begin{document}

\title[]{Infinite-dimensional groups over finite fields and Hall-Littlewood symmetric functions}

\author{Cesar Cuenca and Grigori Olshanski}

\date{}

\begin{abstract}

The groups mentioned in the title are certain matrix groups of infinite size with elements in a finite field $\F_q$. They are built from finite classical groups and at the same time they are similar to reductive $p$-adic Lie groups. In the present paper, we initiate the study of invariant measures for the coadjoint action of these infinite-dimensional groups. Of special interest for us are ergodic invariant measures, which are a substitute of orbital measures.

We examine first the group $\GLB$, a topological completion of the inductive limit group $\varinjlim GL(n,\F_q)$. The traceable factor representations of $\GLB$ were studied by Gorin, Kerov, and Vershik  \cite{GKV}. We show that there exists a parallel theory for ergodic coadjoint-invariant measures, which is linked with harmonic functions on the ``HL-deformed Young graph''  $\Y^\HL(t)$. Here the deformation means that the edges of the Young graph $\Y$ are endowed with certain formal multiplicities depending on the Hall-Littlewood (HL) parameter $t$ specialized to $q^{-1}$.

This fact serves as a prelude to our main results, which concern topological completions of two inductive limit groups built from finite unitary groups, of even and odd dimension. We show that in this case, coadjoint-invariant measures are linked to some new branching graphs. They are still related to the HL symmetric functions, but in a nonstandard way, and the HL  parameter $t$ takes now the negative value $-q^{-1}$. As an application, we find several families of unitarily invariant measures, including analogues of the Plancherel measure.

\end{abstract}

\maketitle

\tableofcontents

\newpage

\section{Introduction}\label{sect1}

\subsection{Setting of the problem}\label{sect1.1}
Fix  a finite field $\F_q$ with $q$ elements and let $\{G(n)\}$ be any of the classical series of finite groups of Lie type over $\F_q$. The basic example is that of general linear groups $G(n)=GL(n,q):=GL(n,\F_q)$, but one can also consider the unitary, orthogonal or symplectic groups over a finite field. Such groups $G(n)$ form a nested chain, so one can form the inductive limit group $G:=\varinjlim G(n)$. The ``infinite-dimensional groups over $\F_q$''  mentioned in the title are certain topological completions $\overline G\supset G$ of such inductive limit groups. 

In the case of general linear groups, the definition of the topological group $\overline G$ is given in the work  \cite{GKV} by Gorin, Kerov, and Vershik. This group, denoted in \cite{GKV} as $\GLB$, is  formed by the infinite size matrices $g=[g_{ij}]_{i,j=1}^\infty$ over $\F_q$ which have finitely many nonzero entries below the diagonal and are invertible. It is a locally compact, separable, totally disconnected topological group with respect to  a natural nondiscrete topology. That topology is uniquely defined by the condition that the subgroup $\mathbb{B}\subset\GLB$ of upper triangular matrices (which is a profinite group) is an open compact subgroup. 

For the other classical series, the corresponding groups $\overline G$ are defined in a similar way. 
Just like it is the case for Lie groups, one can define for them natural analogues of a Lie algebra, adjoint and coadjoint action.

For example, consider the case $\overline G=\GLB$. Then the corresponding Lie algebra $\bar\g$ is formed by the infinite size matrices $X=[X_{ij}]_{i,j=1}^\infty$ over $\F_q$ with finitely many nonzero entries below the diagonal. It is a locally compact vector space over $\F_q$: the topology is defined so that the subspace of upper triangular matrices (which is a profinite group under addition) is an open compact subgroup. Next, let $\bar\g^*$ be the set of infinite size matrices $Y=[Y_{ij}]_{i,j=1}^\infty$ with a modified finiteness condition: $Y$ must have finitely many nonzero elements both below and on the diagonal. The topology on $\bar\g^*$ is defined in a similar way to the topology of $\bar\g$, but now the subspace of \emph{strictly} upper triangular matrices is an open compact subgroup. The topological vector spaces $\bar\g$ and $\bar\g^*$  are dual to each over: the duality between them is given by the bilinear map
$$
\bar\g\times\bar\g^*\to\F, \qquad (X,Y)\mapsto \Tr(XY).
$$
They are also dual to each other as commutative locally compact groups. Finally, note that the group $\overline G=\GLB$ acts on $\bar\g$ and on $\bar\g^*$ by conjugation, and these are the adjoint and coadjoint actions in question. Again, this definition can be extended to the other series.

We raise the following problem.

\begin{problem}\label{problem1.A}
Let $\overline G$ and $\bar\g^*$ be as above.  Study the invariant Radon measures for the coadjoint action of $\overline G$ on $\bar\g^*$. In particular, describe the ergodic invariant Radon measures. 
\end{problem}

Recall that a Radon measure on a locally compact space is a possibly infinite measure taking finite values on compact subsets. In our situation, all invariant measures turn out to be infinite, with a trivial exception. Note also that in the formulation of the problem, the topological group $\overline G$ can be replaced by its countable subgroup $G\subset \overline G$ (the inductive limit subgroup $\varinjlim G(n)$) ---  the set of invariant measures will be the same. 

\begin{remark}\label{remark1.A}

(a) With the only exception of the delta measure at $\{0\}$, the $\overline G$-invariant Radon measures on $\bar\g^*$ are infinite measures and, in particular, cannot be probability measures. Despite this, there are some points of contacts between our theory and Fulman's probabilistic theory of random matrices over finite fields (see his survey paper \cite{F-2002} and references therein). In particular, in both cases, Hall--Littlewood symmetric functions play a fundamental role. 

(b) In the case of a compact group action on a locally compact space, ergodic invariant measures are the same as orbital measures (that is,  invariant measures concentrated on the orbits). Our situation is very different. Namely, there are plenty of orbital measures, related to arbitrary $\overline G$-orbits in $\bar\g^*$; they are all ergodic, but typically fail to be Radon measures. Thus, our ergodic invariant Radon measures are typically not orbital measures. Unlike invariant Radon measures, orbital measures seem to be bad objects that do not admit a reasonable classification.

(c) The Radon condition that we impose on our measures makes it possible to define their Fourier transforms (see Section \ref{sect10}). This is important for the reason explained in subsection \ref{sect1.2} below.

\end{remark}

\subsection{Motivation}\label{sect1.2}

Problem \ref{problem1.A} is prompted by the work of Gorin-Kerov-Vershik \cite{GKV} on unitary representations of the group $\GLB$ (see also the announcements \cite{VK-1998}, \cite{VK-2007}). Since $\GLB$ is not a type I group, classifying its irreducible representations is a wild problem. However, there is a reasonable substitute of irreducible representations --- these are the traceable factor representations. The latter are defined by the indecomposable traces on the subalgebra $\mathcal A(\GLB)\subset L^1(\GLB)$ formed by the locally constant, compactly supported functions on $\GLB$.  

One of the main results of \cite{GKV} is a solution of the classification problem for the indecomposable traces on $\mathcal A(\GLB)$. Here is its brief description: 

(1) the whole set of indecomposable traces can be partitioned into countably many families;

(2) each family is in a natural one-to-one correspondence with the set $\Harm_+(\Y)$ of nonnegative harmonic functions on the Young graph $\Y$.

By definition, the vertex set of the graph $\Y$ (denoted by the same symbol $\Y$) consists of Young diagrams, and the edges are formed by  pairs $\mu\subset \la$ of Young diagrams which differ by a single box (then we write $\mu\nearrow\la)$. The elements of $\Harm_+(\Y)$ are the functions $\vp:\Y\to\R_{\ge0}$ subject to the harmonicity condition
\begin{equation}\label{eq1.B}
\vp(\mu) = \sum_{\la\in\Y:\, \mu\nearrow\la}\vp(\la), \quad \mu\in\Y.
\end{equation}
From this equation, it is seen that $\Harm_+(\Y)$ is a convex cone, and a general theorem asserts that it is isomorphic to the cone of finite measures on a certain space that is associated with the graph and called its \emph{boundary}. The explicit description of the boundary of $\Y$ is known: it is an infinite-dimensional compact space (the Thoma simplex). This is a classical result (equivalent to Thoma's theorem about finite factor representations of the infinite symmetric group). It follows that the indecomposable functions $\vp\in\Harm_+(\Y)$ (that is, the elements of the extreme rays of the cone) correspond to the points of the Thoma simplex, which in turn depend on countably many continuous parameters. This finally leads to an explicit parametrization of the indecomposable traces on $\mathcal A(\GLB)$. 

The work \cite{GKV} raises a number of open problems, one of which is the study of the traces on $\mathcal A(\overline G)$ for other classical series.  

Problem \ref{problem1.A} is in fact a variation of that problem. To see this, we remark that if elements of $\mathcal A(\overline G)$ are treated as test functions, then the traces are precisely the positive definite distributions on $\overline G$, invariant under the action of $\overline G$ by inner automorphisms. On the other hand, the Fourier transforms of the coadjoint-invariant Radon measures on $\bar\g^*$ are precisely the positive definite distributions on the vector space $\bar\g$, invariant under the adjoint action of $\overline G$. In this interpretation, an analogy between the two kinds of objects becomes apparent (in particular, ergodic invariant measures can be treated as a counterpart of indecomposable traces). 
This is a manifestation of a parallelism between problems referring to Lie groups (linked to characters) and to Lie algebras (linked to conjugacy classes), which arises in a great variety of situations, see e.g. \cite{K2}.

It seems that the study of $\overline G$-invariant measures on $\bar \g^*$  can be both easier and more difficult than the study of traces on $\mathcal A(\overline G)$. On the one hand, in contrast to irreducible characters of finite classical groups, the parametrization of conjugacy classes in all finite-dimensional classical Lie algebras over a finite field is achieved by tools of linear algebra (Wall \cite{Wa}, Burgoyne--Cushman \cite{BC}). On the other hand, as can be seen from the comparison of \eqref{eq1.B} and \eqref{eq1.C}, in the case of invariant measures, some of the combinatorial structures that arise (namely the harmonic functions) may be more involved. 

In finite dimensions, the Fourier transforms of invariant functions on  reductive Lie algebras over finite fields were investigated in a number of works, starting with  Springer's note \cite{Springer}  (see e.g. Lehrer \cite{Lehrer}, Letellier \cite{Letellier}, and references therein). We believe that the infinite-dimensional case opens new perspectives in this direction: the distributions on $\bar\g$ obtained as the Fourier transforms of coadjoint-invariant Radon measures seem to be very interesting objects. 

\subsection{The results}

The body of the paper can be divided into two parts. The material of the first part mainly serves us as a guiding example, while the second part contains the main results.

\subsubsection{}

In the first part of the paper (sections 2--5) we examine the case of the general linear groups, which is simpler than that of other classical groups. 
Thus, in Part 1, $\overline G$ is the group $\GLB$. For the space $\bar\g^*$, we use the alternate notation $\La(q)$. Let $\M$ denote the convex cone of invariant Radon measures on $\La(q)$. We show that $\M$ splits into a direct product of countably many convex cones:
\begin{equation}\label{eq1.A}
\M=\prod_{\si\in\Si}\M_\si \qquad \text{($\Si$ is an index set).}
\end{equation}
This decomposition comes from certain partition of $\La(q)$ into a disjoint union of subsets which are invariant and \emph{clopen} (i.e.  both closed and open).  Moreover, all the cones $\M_\si$ are pairwise isomorphic, so we focus on the description of one of them, a distinguished cone denoted by  $\M_0$: it is formed by the invariant measures supported on the  subset $\Nil(\La(q))\subset\La(q)$ of \emph{pronilpotent matrices} (see the definition at the beginning of subsection \ref{sect3.5}).

We show that $\M_0$ is isomorphic to the cone of nonnegative harmonic functions on a  certain \emph{branching graph}, denoted by  $\Y^{\HL}(q^{-1})$.
For a general parameter $t$ ranging over $(0, 1)$, the branching graph $\Y^\HL(t)$ is a $t$-deformation of the Young graph $\mathbb Y$ in the following sense. The two graphs have common vertices and edges, but in $\Y^{\HL}(t)$, the edges are endowed with certain \emph{formal multiplicities}: these are the coefficients $\psi_{\la/\mu}(t)$ in the simplest Pieri rule (multiplication by the first power sum $p_1$) for the Hall--Littlewood (HL) symmetric functions with parameter $t$. In accordance to this, the harmonic functions in this setting satisfy a deformed version of equation  \eqref{eq1.B}, which is obtained by inserting the coefficients $\psi_{\la/\mu}(t)$ on the right-hand side; it takes the form
\begin{equation}\label{eq1.C}
\vp(\mu) = \sum_{\la\in\Y:\, \mu\nearrow\la}\psi_{\la/\mu}(t)\vp(\la), \quad \mu\in\Y,
\end{equation}
see Section \ref{HL_sec} for details. Thus, the picture looks quite similar to what we described above:  the problem is reduced to finding the boundary of a branching graph --- the HL-deformed graph $\Y^\HL(t)$ with $t=q^{-1}$.
Now we can apply results of  Kerov \cite{Ke1} and Matveev \cite{Mat} to complete the classification. Namely, Kerov suggested an ingenious way to construct harmonic functions corresponding to boundary points (in fact, in a more general setting; see Borodin--Corwin \cite[Sect. 2.2.1]{BorC} for details), and Matveev recently managed to prove that Kerov's construction gives exactly all elements of the boundary.
In this way we are able to obtain a complete answer to Problem \ref{problem1.A} in the case of $\overline G=\GLB$.

Note that the central result of part 1, the isomorphism between the cones $\M_0$ and $\Harm_+(\Y^\HL(q^{-1}))$, is essentially equivalent to 
Theorem 4.6 in \cite{GKV} concerning the so-called central measures on the subgroup $\mathbb U\subset\GLB$ of upper unitriangular matrices. However, we present the material from another perspective and also give a number of other results.

\subsubsection{}

In the second part of the paper (sections 6--9) we examine the case of unitary groups. The field $\F_q$ has a unique quadratic extension $\F_{q^2}$, which makes it possible to define sesquilinear Hermitian forms. In each dimension $N$, there is only one, within equivalence, nondegenerate sesquilinear Hermitian form, and we denote by  $U(N,q^2)$ the corresponding unitary group, i.e., the group of linear transformations preserving this form.
It can be realized as a subgroup of $GL(N,q^2)$ in various ways, depending on the choice of the matrix of the form. For building a $\GLB$-like topological completion $\overline G$ of an inductive limit group $G=\varinjlim G(n)$, we need embeddings $G(n)\to G(n+1)$ which are consistent with Borel subgroups. To satisfy this condition, we have to consider separately two twin series $\{G(n)\}$ and two direct limit unitary groups, which we call informally \emph{even} and \emph{odd} (this term refers to the parity of $N$):
\begin{equation*}
U(2\infty,q^2):=\varinjlim U(2n,q^2), \qquad U(2\infty+1,q^2):=\varinjlim U(2n+1,q^2).
\end{equation*}
In both series, the matrices of sesquilinear forms have $1$'s along the secondary diagonal and $0$'s elsewhere --- this leads to the required embeddings of groups and makes it possible for us to define the two desired topological completions.

Next, we define, in a natural way, the two corresponding Lie algebras, their dual spaces, and the coadjoint actions. Let $\M^{\e}$ and $\M^{\o}$ be the cones formed by the coadjoint-invariant Radon measures (here $\e$ and $\o$ are abbreviations of \emph{even} and \emph{odd}, respectively). We obtain direct product decompositions analogous to \eqref{eq1.A}, with another countable index set $\wt\Si$,
\begin{equation*}
\M^{\e} = \prod_{\si\in\wt\Si}\M^{\e}_{\si}, \qquad \M^{\o} = \prod_{\si\in\wt\Si}\M^{\o}_{\si},
\end{equation*}
as well as two distinguished cones $\M^{\e}_0$ and $\M^{\o}_0$, see Subsection \ref{general_measures_U} (the notation there is a bit different). 

Here a new effect arises: in the first decomposition, some components are isomorphic to $\M^{\e}_0$ and the other are isomorphic to $\M^{\o}_0$, and likewise for the second decomposition. Thus the two versions, even and odd, are intertwined. But we obtain again a reduction of our problem: it suffices to  study the distinguished cones $\M^{\e}_0$ and $\M^{\o}_0$.

In Theorem \ref{thm8.A}, we find two new branching graphs: $\Y^\HL_\e(t)$ and $\Y^\HL_\o(t)$. They are linked with the HL symmetric functions with the \emph{negative} parameter $t\in(-1,0)$. The vertices of $\Y^\HL_\e(t)$ are the Young diagrams of even size, and those of $\Y^\HL_\o(t)$ are the Young diagrams of odd size. The formal edge multiplicities depend on $t$ and are defined from the multiplication by $p_2$ in the HL basis. Theorem \ref{Ustructure} claims that the cones $\M^{\e}_0$ and $\M^{\o}_0$ are isomorphic to the cones of nonnegative harmonic functions on $\Y^\HL_\e(-q^{-1})$ and $\Y^\HL_\o(-q^{-1})$, respectively. 

Theorems \ref{thm8.A} and  \ref{Ustructure} are the main results of the present paper (the main computation used in the proof of Theorem \ref{Ustructure} is deferred to Section \ref{sect9}). The construction of Theorem \ref{thm8.A} is unusual in that the formal edge multiplicities are defined through the multiplication by $p_2$ instead of $p_1$. Another novel phenomenon is the appearance of the Hall--Littlewood functions with negative parameter $t$ (it strongly resembles \emph{Ennola's duality} \cite{E1}, \cite{E2}, \cite{TV}). 

As a first application of the main results, we construct a few examples of invariant measures including an analogue of the Plancherel measure (Section \ref{sec:plancherel}). 

We can construct a large family of invariant measures by different tools (this will be the subject of a future publication), but at present we do not have a complete classification. 

In view of the above, Problem \ref{problem1.A} for the even and odd unitary groups reduces to finding the boundaries of the new branching graphs $\Y^\HL_\e(t)$ and $\Y^\HL_\o(t)$ with $t=-q^{-1}$. The latter problem in turn can be formulated more broadly:

\begin{problem}\label{problem1.B}
Find the boundaries of the even and odd HL-deformed branching graphs $\Y^\HL_\e(t)$ and $\Y^\HL_\o(t)$ with parameter $t\in(-1,0)$.
\end{problem} 

\subsubsection{} The last Section \ref{sect10} may be viewed as an appendix. Here we present a simple general result linking coadjoint-invariant Radon measures on $\bar\g^*$ to generalized spherical  representations of the semidirect product group $\overline G\ltimes\bar\g$. (In Lie theory, such semidirect products are called \emph{Takiff groups}.)

\subsection{Acknowledgements}

We are grateful to the referee for the careful reading of the manuscript.
This project started while the first author (C.~C.) worked at California Institute of Technology. The research of the second author (G.~O.) was supported by the Russian Science Foundation, project 20-41-09009.

\section{The finite general linear group: preliminaries}\label{sect2}

\subsection{The group $GL(n,q)$ and the Lie algebra $\gl(n,q)$}\label{sect2.1}

Let $q$ be the power of a prime number.
Fix the finite field $\F:=\mathbb F_q$ with $q$ elements.
Denote by $GL(n, q)$ the group of invertible $n\times n$ matrices with entries in $\F$.
Denote by $\Mat_n(q) = \Mat_n(\F)$ the associative algebra of square matrices of order $n$, over the finite field $\F$; the corresponding Lie algebra, with commutator $[X, Y] = XY - YX$, is denoted $\gl(n, q)$.
The group $GL(n, q)$ is called the \emph{general linear group} over $\F$, and $\gl(n, q)$ acts as its Lie algebra.
The adjoint action of $GL(n, q)$ on $\gl(n, q)$ is matrix conjugation. The invariant bilinear form $(X,Y)\mapsto \Tr(XY)$ on $\gl(n,q)$ allows one to identify the vector space $\gl(n,q)$ with its dual, so we may identify the adjoint and coadjoint actions of $GL(n,q)$.

Let us agree that $GL(0, q) := \{1\}$ is the trivial group and $\gl(0, q) := \{0\}$ is the zero Lie algebra.

\subsection{Partitions}\label{sect2.2}
We recall some notions related to integer partitions, which  will be used throughout the paper. 

A partition is an infinite sequence $\la = (\la_1, \la_2, \cdots)$ of nonnegative integers such that $\la_1 \geq \la_2 \geq \dots \geq 0$ and only finitely many terms are distinct from $0$. We use lowercase Greek letters to denote partitions, e.g. $\la, \mu$. The size of the partition $\la$ is $|\la| := \sum_{i\geq 1}{\la_i}$. Further, we will use the following standard notation (Macdonald \cite[Ch. I]{Mac}):
\begin{equation}\label{eq2.A}
\ell(\la) := \max\{ j \mid \la_j \neq 0 \},\quad  n(\la) := \sum_{i \geq 1}{(i-1)\la_{i}},\quad m_i(\la) := \#\{ j \geq 1 \mid \la_j = i \}, \ i\geq 1.
\end{equation}

Following \cite{Mac} we identify partitions with their Young diagrams. We denote by $\Y$ the set of all partitions (=Young diagrams) and write $\Y_n$ for the set of partitions of size $n$. In particular,  $\Y_0$ only contains the zero partition (= empty Young diagram), to be denoted $\emptyset$.

The partition $\la' = (\la_1', \la_2', \cdots)$ corresponding to the transposed Young diagram $\la'$ is given by $\la_k' := \sum_{j \geq k}{m_j(\la)}$, for all $k\geq 1$.

\subsection{Nilpotent conjugacy classes in $\gl(n, q)$}\label{orbits_GL}\label{sect2.3}

By a conjugacy class in $\gl(n,q)$ we mean a $GL(n,q)$-orbit in this space. A conjugacy class will be called nilpotent if it consists of nilpotent matrices. The set of nilpotent matrices in $\gl(n,q)$ will be denoted by $\Nil(\gl(n, q))$.  

We say that  a nilpotent matrix $X\in\Nil(\gl(n, q))$ has \emph{Jordan type $\la\in\Y_n$} if its Jordan blocks have lengths $\la_1,\la_2,\cdots$\,. This gives a parametrization of nilpotent conjugacy classes in $\gl(n,q)$ by partitions of size $n$. The class corresponding to a partition $\la\in\Y_n$ will be denoted by $\{\la\}$. In this notation,  the decomposition of $\Nil(\gl(n, q))$ into conjugacy classes is written as
\begin{equation*}
\Nil(\gl(n, q)) = \bigsqcup_{\la\in\Y_n}{\{\la\}}.
\end{equation*}

\subsection{General conjugacy classes in $\gl(n,q)$}\label{sect2.4}

Let $\T_n$ denote the set of conjugacy classes in $\gl(n,q)$. A matrix $X\in\gl(n, q))$ belonging to a class $\tau\in\T_n$ is said to be \emph{of type $\tau$}. 

Let $\NSin(\gl(n,q))\subset\gl(n,q)$ be the subset of nonsingular matrices (that is, matrices with nonzero determinant). It is $GL(n,q)$-invariant. Let  $\Si_n\subset\T_n$ be the subset of conjugacy classes contained in $\NSin(\gl(n,q))$. Elements of $\Si_n$ will be called \emph{nonsingular classes} or \emph{nonsingular types}.

\begin{lemma}\label{lemma2.A}
There is a natural bijection 
$$
\T_n\leftrightarrow \bigsqcup_{s=0}^n (\Si_s\times\Y_{n-s}).
$$
\end{lemma}

(Here we regard $\Si_0$ as a singleton, so that $\Si_0\times\Y_n$ is identified with $\Y_n$.) 

\begin{proof}
It is convenient to regard matrices  $X\in \gl(n,q)$ as linear operators on the vector space $V:=\F^n$.  For any $X$, there is a unique direct sum decomposition $V=V'\oplus V_0$ such that both $V'$ and $V_0$ are $X$-invariant, $X\big|_{V'}$ is invertible,  and $X\big|_{V_0}$ is nilpotent. Because of uniqueness, this decomposition is $GL(n,q)$-invariant. It provides the desired bijection. 
\end{proof} 

By virtue of the lemma, each $\tau\in\T_n$ is represented by a pair $(\si,\la)$, where $\si\in\Si_s$ and $\la\in\Y_{n-s}$ for some $s$, $0\le s\le n$. 

Thus, the parametrization of general conjugacy classes in $\gl(n,q)$ is reduced to the description of the sets $\Si_s$. The latter is given in the remark below, but in fact we do not need it. We will only use the bijection established in the lemma.

\begin{remark}[cf. Macdonald \cite{Mac}, Ch. IV, Sect. 2, or Burgoyne-Cushman \cite{BC}]\label{rem2.A}
Let $\Phi'$ denote the set of irreducible, monic polynomials in $\F[x]$ with nonzero constant term. There is a bijective correspondence between elements of $\Si_s$ and maps $\boldsymbol{\mu}: \Phi'\to \Y$ such that $\boldsymbol{\mu}(f) = \emptyset$ for all but finitely many polynomials $f\in\Phi'$ and
$$
\sum_{f\in\Phi'}\deg(f)|\boldsymbol{\mu}(f)|=s,
$$
where $\deg(f)$ is the degree of $f$.
\end{remark}

\section{Invariant measures on $\La(q)$: generalities}\label{sec_invmeasures}

\subsection{The space $\La(q)$. The groups $GL(\infty,q)$ and $\GLB$}\label{sect3.1}

Let $\Mat_{\infty}(q) = \Mat_{\infty}(\F)$ be the space of infinite matrices $M = [m_{i, j}]_{i, j \geq 1}$ with entries in $\F$.
For $n\in\Z_{\geq 0}$, let $\La_n(q)\subset\Mat_{\infty}(q)$ be the subset of matrices $M\in\Mat_{\infty}(q)$ such that $m_{i, j} = 0$, whenever $i \geq j$ and $i > n$.
In particular, $\La_0(q)$ is the set of strictly upper triangular matrices.
We extend the definition of $\La_n(q)$ to negative values of $n$ as follows: if $n\in\Z_{<0}$, then $\La_{n}(q)\subset \La_0(q)$ is the subgroup formed by the matrices for which the first $|n|$ columns are null.

As an example, matrices in $\La_2(q)$ and in $\La_{-2}(q)$ look like
$$
M=\begin{bmatrix} * & * & * &  * & * &  \\ * & * & * & * & * & \vdots \\ 0 & 0 & 0 &  * & * &  \\ 0 & 0 & 0 & 0 & * & &  \\  & \cdots & & &  & \ddots \end{bmatrix} \quad \text{and} \quad M=\begin{bmatrix} 0& 0 & * &  * & * &  \\ 0 & 0 & * & * & * & \vdots \\ 0 & 0 & 0 &  * & * &  \\ 0 & 0 & 0 & 0 & * & &  \\  & \cdots & & &  & \ddots \end{bmatrix}\!,
$$
respectively (here an asterisk stands for an arbitrary element of $\F$).

Each set $\La_n(q)$ is a vector space over $\F$ and, in particular, a commutative group under addition.
Note that $\La_n(q)$ is contained in $\La_{n+1}(q)$ as a subgroup of finite index $q^{|n+1|}$, for every $n\in\Z$.

\begin{definition}
Let $\La(q)$ be the inductive limit group $\varinjlim \La_n(q)$. As a set, it is the union of all $\La_n(q)$. In other words, $\La(q)$ is the set of almost strictly upper triangular matrices.
\end{definition}

Each $\La_n(q)$ is a profinite group (a projective limit of finite groups); as such it is endowed with the corresponding projective limit topology making it a compact topological group. Equivalently, the topology of $\La_n(q)$ is the topology of pointwise convergence of matrix elements. Evidently, $\La_n(q)$ is an open subgroup of $\La_{n+1}(q)$ for each $n\in\Z$. Next, we equip $\La(q)$ with the inductive limit topology; then it becomes a locally compact topological group. The topology of $\La(q)$ is totally disconnected:  each subgroup $\La_n(q)$ is compact and clopen (both open and closed), and the subgroups $\La_n(q)$ with $n < 0$ form a fundamental system of neighborhoods of $\{0\}$.

The natural inclusions $GL(n, q)\hookrightarrow GL(n+1, q)$ give rise to the inductive limit group $GL(\infty, q) := \varinjlim GL(n, q)$. It is a countable group and it acts on $\La(q)$ by conjugations.

Recall that the group $\GLB$ is formed by invertible infinite size matrices over $\F$ with finitely many entries below the diagonal. As in the case of $\La(q)$, in the group $\GLB$ there is a doubly infinite chain $\{\GLB_n: n\in\Z\}$ of subgroups such that 
$$
\GLB_n\subset \GLB_{n+1}, \quad \bigcap_n \GLB_n=\{e\}, \quad \bigcup_n \GLB_n=\GLB,
$$
and each $\GLB_n$ is a profinite group and has finite index in $\GLB_{n+1}$.
Here $\GLB_0$ is the ``Borel subgroup'' $\mathbb B\subset\GLB$ formed by invertible upper triangular matrices; if $n\ge1$, then $\GLB_n$ is the ``parabolic subgroup'' formed by matrices with $0$'s at all positions $(i,j)$ such that $i>n$ and $j<i$; finally, if $n<0$, then $\GLB_n$ is the subgroup of $\mathbb B$ formed by matrices whose upper left block of size $|n|\times|n|$ is the identity matrix. Again, $\GLB$ is a totally disconnected, locally compact topological group; its topology is uniquely determined by the condition that each $\GLB_n$ is a compact and open subgroup. Note that $GL(\infty,q)$ is a dense subgroup of $\GLB$.

Like $GL(\infty,q)$, the group $\GLB$ acts on $\La(q)$ by conjugation. We regard this as the \emph{coadjoint action} (see Section \ref{sect1.1}).  

\begin{remark} 
One can prove that the action $\GLB\times\La(q)\to\La(q)$ is continuous. We do not use this fact anywhere except at the very end of the illustrative Section \ref{sect10}. Note also that it is this fact that ensures the existence of the orbital measures, which we spoke about in the introduction. 
\end{remark}

\subsection{Radon measures on $\La(q)$}

For any $M\in\Mat_{\infty}(q)$, denote $M^{\{n\}} := [m_{i, j}]_{i, j = 1, \cdots, n}$ its upper-left $n\times n$ corner; it is a matrix from $\gl(n, q)$.   An \emph{elementary cylinder set of level\/ $n\in\Z_{\ge0}$} is a subset of $\La(q)$ of the form
\begin{equation}\label{N_subsets}
\Cyl_n(X):=\{ M \in \La(q) \mid M\in \La_n(q),\ M^{\{n\}} = X \}\subset\La(q),
\end{equation}
where $n\in\Z_{\geq 0}$ and $X\in\gl(n, q)$. If we denote by $X_0$ the unique element from the zero Lie algebra $\gl(0, q)$, we agree that $\Cyl_0(X_0) := \La_0(q)$. 

More generally, a \emph{cylinder set of level\/ $n$} is, by definition, a (finite) union of some elementary cylinder sets of the same level. In particular, $\La_n(q)$ is a cylinder set of level $n$. Note that any cylinder set of level $n$ is also a cylinder set of level $n+1$. In particular, for elementary cylinders we have
$$
\Cyl_n(X)=\bigsqcup_{Y}\Cyl_{n+1}(Y),\quad X\in\gl(n, q),
$$
where the union is over matrices $Y\in\gl(n+1, q)$ whose upper-left $n\times n$ corner is $X$, and whose $(n+1)$-th row has all zeroes.

The elementary cylinder sets are clopen compact sets; they form a base of the topology of $\La(q)$. 

By a \emph{measure} we always mean a countably additive nonnegative set function (we will not need complex or signed measures).  Recall that a \emph{Radon measure} on a locally compact space is a Borel measure (possibly of infinite mass) with the property that it takes finite values on compact subsets. 

\begin{lemma}\label{lemma3.A}
There is a bijective correspondence between Radon measures on $\La(q)$ and finitely-additive, nonnegative set functions on the  cylinder sets of all levels, with finite values. 
\end{lemma}

\begin{proof}
Any Radon measure obviously produces a finitely-additive function on the cylinder sets. Let us show that, conversely, any finitely-additive function on the ring of cylinder sets admits a unique extension to a Radon measure. If an extension exists, then it is unique, because the cylinder sets generate the sigma algebra of Borel sets. To show the existence, we are going to apply the Carath\'eodory extension theorem from measure theory. Its hypothesis is satisfied for the trivial reason that any decomposition of a cylinder set into disjoint open subsets must be \emph{finite} (because all cylinder sets are compact). Thus, the theorem is applicable. It gives us a true (that is, countably-additive) measure on $\La(q)$. Finally, we claim that it satisfies the  Radon condition. Indeed, this follows from the fact that any  compact subset of $\La(q)$ is contained in $\La_n(q)$ with $n$ large enough. This fact, in turn, is a particular case of the following general result: let $X_1\subset X_2\subset \dots$ be an ascending sequence of locally compact Hausdorff spaces such that each $X_n$ is a closed subspace of $X_{n+1}$, and let $X:=\varinjlim X_n$ be their union, endowed with the inductive limit topology; then any compact subset of $X$ is contained in some $X_n$.  For a proof, see e.g. Gl\"ockner \cite[Lemma 1.7, item (d)]{Gloeckner}.
\end{proof}

\subsection{Invariant measures on $\La(q)$}

\begin{definition}
Define $\PGL$ as the set of $GL(\infty, q)$-invariant Radon measures on $\La(q)$.
\end{definition}

\begin{proposition}\label{prop3.A}
Any measure $P\in\PGL$ is automatically $\GLB$-invariant.
\end{proposition}

\begin{proof}
Let $g\in\GLB$ be arbitrary. It suffices to prove that for any elementary cylinder set $C\subset \La(q)$, one has $P(gCg^{-1})=P(C)$. Further, with no loss of generality, it suffices to prove this under the additional assumption that $n$,  the level of $C$,  is large enough. We will assume that $n$ is so large that all matrix entries $g_{ij}$, such that $i>j$ and $i>n$, are equal to $0$. Let  $h:=g^{\{n\}}$ be the upper left $n\times n$ corner of the matrix $g$. Our assumption implies that $h$ is invertible and hence belongs to $GL(n,q)$. Next, the conjugation by $g$ permutes the elementary cylinder sets of level $n$: if $X\in\gl(n,q)$, then
$$
g \Cyl_n(X) g^{-1}= \Cyl_n(hXh^{-1}).
$$
Because $P$ is $GL(\infty,q)$-invariant, we have $P(\Cyl_n(hXh^{-1}))=P(\Cyl_n(X))$. This completes the proof.
\end{proof}

Note that $\PGL$ is a convex cone. Its description will be given in Section \ref{sect5}. In particular, we will describe the extreme measures in the sense of the following definition.  

\begin{definition}
A measure $P\in\PGL$ is said to be \emph{extreme} if it is nonzero and if any other measure $P'\in\PGL$ which is majorated by $P$ (meaning that $P'(A) \le P(A)$ for all Borel sets $A\subseteq\La(q)$) is in fact proportional to $P$. Equivalently, if $P'\in\PGL$ is absolutely continuous with respect to $P$, then $P'$ is proportional to $P$.
\end{definition}

A standard argument shows that $P\in\PGL$ is extreme if and only if it is \emph{ergodic} in the sense that  if $A\subset \La(q)$ is a $GL(\infty,q)$-invariant Borel subset, then is either $A$ is $P$-null or its complement $\La(q)\setminus A$ is $P$-null. See Phelps \cite[Proposition 12.4]{Phelps} (although Phelps considers probability measures, the claim remains valid for infinite measures as well). 

\subsection{All nontrivial invariant Radon measures on $\La(q)$ are infinite}
By a trivial measure we mean a multiple of the Dirac measure at the point $0\in\La(q)$. 

\begin{proposition}
All nontrivial measures $P\in\PGL$ are infinite measures.  
\end{proposition}

In fact, this claim holds under a weaker assumption: it suffices to assume that $P$ is invariant under the action of the infinite symmetric group $S(\infty):=\varinjlim S(n)$ embedded into $GL(\infty,q)$ in the natural way.

\begin{proof}
Suppose $P$ is an $S(\infty)$-invariant Borel measure on $\La(q)$ with finite total mass; we shall show that $P$ is concentrated at $\{0\}$. We may assume that the total mass equals $1$.  Regard $(\La(q), P)$ as a probability space. Then the entries $m_{ij}$ of matrices $M\in\La(q)$ turn into $\F$-valued random variables; let us denote them by $\xi_{ij}$.
Consider the collection of random variables $\eta_i:=\xi_{i+1,1}$, where $i=1,2,3,\dots$.  This collection is \emph{exchangeable}, because any finitary permutation of them can be implemented by a suitable permutation matrix lying in $S(\infty)$.
Therefore we may apply de Finetti's theorem, which tells us that $\{\eta_i\}$ is a mixture of i.i.d random variables.
On the other hand, by the very definition of $\La(q)$, the number of nonzero $\eta_i$'s is finite, with $P$-probability $1$.  This is only possible if all $\eta_i$'s are in fact equal to $0$, with $P$-probability $1$. Next, every variable $\xi_{ij}$ with $i\ne j$ can be transformed into $\eta_1=\xi_{21}$ by means of conjugation by a matrix from $S(\infty)$. It follows that $\xi_{ij}=0$ for all $i\ne j$, with $P$-probability $1$. Hence $P$ is concentrated on the subset of diagonal matrices. Finally, we repeat the same argument, based on de Finetti's theorem, to the random variables $\xi_{ii}$: it equally follows that $\xi_{ii} = 0$ for all $i$, with $P$-probability $1$, thus concluding the proof of our initial claim.
\end{proof}

\subsection{Invariant measures on pronilpotent matrices}\label{sect3.5}

Let $\Nil(\La_n(q)) \subset \La_n(q)$ be the set of matrices $M\in \La_n(q)$ with nilpotent $n\times n$ corners $M^{\{n\}}$.  From the definition of $\La_n(q)$ it follows that $\Nil(\La_n(q))\subset\Nil(\La_{n+1}(q))$. Also, let
$$
\Nil(\La(q)) := \bigcup_{n = 1}^{\infty}{\Nil(\La_n(q))}.
$$
This set consists of matrices  $M\in\La(q)$ for which any sufficiently large corner $M^{\{n\}}$ is nilpotent.
Matrices from $\Nil(\La(q))$ will be called \emph{pronilpotent matrices}. 

\begin{lemma}\label{lemma3.nil}
One has 
\begin{equation*}
\Nil(\La(q))=\bigcup_{g\in GL(\infty,q)} g\, \La_0(q) g^{-1}.
\end{equation*}
\end{lemma}
\begin{proof}
This follows from the similar claim for finite dimension: any nilpotent matrix $X\in\gl(n,q)$ is conjugated to a strictly upper triangular matrix. 
\end{proof}

It follows, in particular, that the set $\Nil(\La(q))$ is open and $GL(\infty,q)$-invariant. (It is $\GLB$-invariant, too.)

\begin{definition}\label{supp_pro}
Define $\PGL_0\subset\PGL$ as the subset of measures which are supported on $\Nil(\La(q))$.
Further, let $\B_0^{GL(\infty, q)}\subset\PGL_0$ be the subset of measures $P$ that are normalized by the condition $P(\La_0(q)) = 1$.
\end{definition}

The set $\PGL_0$ has the structure of a convex cone. Note that its extreme rays are also extreme rays of the ambient cone $\PGL$.
From Lemma \ref{lemma3.nil} it follows that if $P$ is a nonzero measure from $\PGL_0$, then $P(\La_0(q))>0$ (here we use the fact that $GL(\infty,q)$ is countable). This in turn implies that $\B_0^{GL(\infty, q)}$ is a base of the cone $\PGL_0$. 

In the next section, we give a description of measures from the cone $\PGL_0$. Since any nonzero measure from $\PGL_0$ can be uniquely expressed as $c\cdot P$, for some $P\in\B^{GL(\infty, q)}_0$ and $c>0$, our results also represent a description of measures from $\B^{GL(\infty, q)}_0$.
Later, in Section \ref{sect5}, we deduce from this a description of the cone $\PGL$ of general invariant measures on $\La(q)$.  

\begin{remark}\label{rem3.A}
Following \cite{GKV}, denote by $\mathbb U\subset\mathbb B$ the subgroup formed by the upper-unitriangular matrices. There is a natural bijection $\La_0(q)\to\mathbb U$ that assigns to each $M\in\La_0(q)$ the invertible matrix $1+M$. Given a measure $P\in\B^{GL(\infty, q)}_0$, the pushforward of $P\big|_{\La_0(q)}$ with respect to the bijection $\La_0(q)\to\mathbb U$ is a \emph{central probability measure} on $\mathbb U$ in the sense of \cite[Definition 4.3]{GKV}. In this way we obtain a bijective correspondence between $\B^{GL(\infty, q)}_0$ and the set of central probability measures on $\mathbb U$.  However, central measures on $\mathbb U$ catch only measures on pronilpotent matrices and leave aside the more general invariant measures on $\La(q)$ considered in Section \ref{sect5}.  
\end{remark}

\section{Description of $\PGL_0$}\label{sect4}

Describing $\PGL_0$ is equivalent to describing the set of nonnegative harmonic functions on certain branching graph. After making this connection, we invoke the recently proved Kerov's conjecture (Matveev \cite{Mat}) to obtain the desired description (Theorem \ref{thm4.A}).

Our construction of ergodic normalized measures supported on pronilpotent matrices essentially coincides with the construction of ergodic central measures on $\mathbb U$, as given in \cite[Theorem 4.6]{GKV}.

\subsection{Generalities on branching graphs}

We shall consider several branching graphs and the space of harmonic functions on them. The general notions, at the level of generality that we will need, are encapsulated in the following definition.

\begin{definition}\label{def4.A1}
By a \emph{branching graph} $\Ga$ we mean a graph with graded vertex set $\bigsqcup_{n=0}^\infty\Ga_n$ and formal edge multiplicities (or \emph{weights}), subject to the following conditions:

--- all levels $\Ga_n$ are finite nonempty sets and $\Ga_0$ is a singleton;

--- only vertices of adjacent levels can be joined by an edge;

--- each vertex of level $n\in\Z_{\geq 0}$ is joined with at least one vertex of level $n+1$;

--- each vertex of level $n\in\Z_{\geq 1}$ is joined with at least one vertex of level $n-1$; 

--- all edges are simple and their weights are strictly positive real numbers. 
\end{definition}

With a slight abuse of notation, we use the same symbol $\Ga$ to denote both a branching graph and its vertex set. 

\begin{example} Recall the definition of the \emph{Young graph} $\Y$: its vertices are arbitrary partitions (=Young diagrams) and the edges are formed by pairs $\mu\nearrow\la$  of diagrams, where
$$
\mu\nearrow \la \; \Leftrightarrow\; \text{$|\la|=|\mu|+1$ and $\mu\subset\la$}.
$$ 
By definition, all edge weights of $\Y$ equal $1$. In this section, we will deal with certain branching graphs which differ from $\Y$ by a different system of edge weights. In Section \ref{sec:ennola} some new graphs will appear.
\end{example}

\begin{definition}
Let $\Ga$ be a branching graph. Given two vertices, $v_n\in\Ga_n$ and $v_{n+1}\in\Ga_{n+1}$,  forming an edge, let $W^{n+1}_n(v_{n+1},v_n)$ stand for the corresponding weight.  

(1) A \emph{harmonic function} on $\Ga$ is a real-valued function $\vp$ on $\Ga$  such that
$$
\vp(v_n) = \sum_{v_{n+1}\in \Ga_{n+1}}{W^{n+1}_n(v_{n+1}, v_n) \vp(v_{n+1})}, \quad v_n\in \Ga_n,\ n\in\Z_{\geq 0}.
$$

(2) The set of \emph{nonnegative} harmonic functions on $\Ga$ will be denoted by $\Harm_+(\Ga)$. Note that it is a convex cone. 

(3) Let $v_0$ denote the \emph{root} of $\Ga$ ---  the only vertex of level $0$. The subset of functions $\vp\in\Harm_+(\Ga)$ normalized by the condition $\vp(v_0)=1$ will be denoted by $\Harm_1(\Ga)$. Note that it is a convex set, which serves as a base of the cone $\Harm_+(\Ga)$.  

(4) The set $\ex(\Harm_1(\Ga))$ of \emph{extreme points} of the convex set $\Harm_1(\Ga)$ is called the \emph{boundary} of $\Ga$. 
\end{definition}

Let $\R^\Ga$ be the space of arbitrary real-valued functions on $\Ga$, equipped with the topology of pointwise convergence. 

\begin{proposition}\label{prop4.A}
The followings claims hold true.

{\rm(1)} $\Harm_1(\Ga)$ is a nonempty compact subset of\/ $\R^\Ga$.

{\rm(2)} The boundary $\ex(\Harm_1(\Ga))\subset \Harm_1(\Ga)$ is a nonempty subset of type $G_\de$, hence a Borel set.

{\rm(3)} There exists a bijective correspondence $\vp\leftrightarrow \m$ between functions $\vp\in\Harm_1(\Ga)$ and probability Borel measures $\boldsymbol{m}$ on the boundary $\ex(\Harm_1(\Ga))$, given by
$$
\vp(v)=\int_{\psi\in \ex(\Harm_1(\Ga))}\psi(v) \boldsymbol{m}(d\psi), \qquad v\in\Ga.
$$
\end{proposition}

This formula also establishes an isomorphism between the cone $\Harm_+(\Ga)$ and the cone of finite Borel measures on the boundary. 

\begin{proof}
Claim (1) is evident. For claim (2), see Phelps \cite[Proposition 1.3]{Phelps} (the fact that the boundary is nonempty follows from the Krein--Milman theorem). Claim (3) can be deduced from Choquet's theorem \cite[Sect. 10]{Phelps}: see, e.g. \cite[Theorem 9.2]{Ols-2003a}.
\end{proof}

\begin{definition}\label{def4.sim}
Let $\Ga$ and $\Ga'$ be two branching graphs with common sets of vertices and edges, but different systems of edge weights, $\{W^{n+1}_n\}$ and $\{(W')^{n+1}_n \}$. Following Kerov \cite[Ch. 1, Sect. 2.2]{Ke2}, we say that $\Ga$ and $\Ga'$ are \emph{similar} if there exists a positive real-valued function $f$ on the vertex set, such that
\begin{equation}\label{sim_weights}
(W')^{n+1}_n(v_{n+1}, v_n) = W^{n+1}_n(v_{n+1}, v_n) \cdot \frac{f(v_{n})}{f(v_{n+1})}.
\end{equation}
Then $f$ is called the \emph{gauge function from $\Gamma$ to $\Gamma'$}. With no lost of generality we may (and will) assume that $f(v_0)=1$.
\end{definition}

The following lemma is used several times. We omit the proof.

\begin{lemma}\label{homol_graphs}
The convex cones of nonnegative harmonic functions on two similar branching graphs are affine-isomorphic. More explicitly, if\/ $\Gamma$ is similar to $\Gamma'$ and $f$ is the gauge function from $\Gamma$ to $\Gamma'$, then $\vp \mapsto\vp' := \vp\cdot f$ defines an affine-isomorphism $\Harm_+(\Gamma) \stackrel{\cong}{\to}\Harm_+(\Gamma')$ as well as 
an isomorphism of convex sets $\Harm_1(\Gamma) \stackrel{\cong}{\to}\Harm_1(\Gamma')$.
\end{lemma}

\subsection{HL-deformed  Young graph $\Y^{\HL}(t)$ and its boundary}\label{HL_sec}

Recall some terminology and results from Macdonald \cite[Ch. III]{Mac}.
Let $t$ be any real number in the interval $(0, 1)$.
Denote by $\Sym$ the graded $\R$-algebra of symmetric functions.
It is known that $\Sym$ is freely generated by $1$ and the power-sums $p_1, p_2, \cdots$.
Moreover, $\Sym$ has a basis with elements parametrized by partitions $\la\in\Y$, consisting of the \emph{Hall-Littlewood functions} $P_{\la}(; t)$.
For conciseness, we shall abbreviate Hall-Littlewood by HL.
The HL function $P_{\la}(; t)$ is homogeneous of degree $|\la|$.
We also need the $Q$-version of the HL functions, denoted $Q_{\la}(; t)$:
\begin{equation}\label{b_def}
Q_{\la}(; t) := b_{\la}(t) P_{\la}( ; t),\qquad b_{\la}(t) := \prod_{i \geq 1}{\prod_{1 \leq j \leq m_i(\la)}{(1 - t^j)}}.
\end{equation}

The following is the simplest \emph{Pieri rule} for HL functions (see \cite[Ch. III, (5.7$'$), (5.8$'$)]{Mac}):
\begin{equation}\label{pieri_HL}
(1-t)p_1 \cdot Q_\mu(;t) = \sum_{\la :\, \mu\nearrow\la}\psi_{\la/\mu}(t) Q_\la(;t),
\end{equation}
where the coefficients $\psi_{\la/\mu}(t)$ are defined as follows: let $k = k(\la/\mu)$ be the column number of the single box that differs $\la$ from $\mu$; then 
\begin{equation}\label{psi_def}
\psi_{\la/\mu}(t) := \begin{cases}
1 - t^{m_{k-1}(\mu)} & \text{if }k > 1,\\
1 & \text{if }k = 1.
\end{cases}
\end{equation}
Observe that all coefficients $\psi_{\la/\mu}(t)$ are strictly positive.

It is convenient for us to rewrite \eqref{pieri_HL} in a slightly modified form:
\begin{equation}\label{eq4.D}
p_1 \cdot \frac{Q_\mu(;t)}{(1-t)^{|\mu|}} = \sum_{\la :\, \mu\nearrow\la}\psi_{\la/\mu}(t) \frac{Q_\la(;t)}{(1-t)^{|\la|}}.  
\end{equation}

\begin{definition}\label{def4.HL}
The \emph{HL-deformed Young graph $\Y^{\HL}(t)$} with parameter $t\in(0,1)$ is the branching graph with the same vertex and edge sets as in the Young graph $\Y$ and with the weights $\psi_{\la/\mu}(t)$ assigned to the edges $\mu\nearrow\la$. 
\end{definition}

Thus, harmonic functions $\vp$ on $\Y^\HL$ are defined by the relations
\begin{equation}\label{coherenceGL}
\vp(\mu) = \sum_{\la :\, \mu\nearrow\la}\psi_{\la/\mu}(t)\vp(\la),\quad\mu\in\Y.
\end{equation}

\begin{remark}\label{rem4.A}
There is a one-to-one correspondence $\varphi\leftrightarrow \Phi$ between harmonic functions $\vp\in\Harm_1(\Y^\HL(t))$ and linear functionals $\Phi:\Sym\to\R$ satisfying the conditions
\medskip

$\bullet$ $\Phi(p_1F)=\Phi(F)$ for any $F\in\Sym$ (harmonicity);
\smallskip

$\bullet$ $\Phi$ is nonnegative on the convex cone $C^\HL(t)\subset\Sym$ spanned by the functions $Q_\la(;t)$, $\la\in\Y$ (positivity);
\smallskip

$\bullet$ $\Phi(1)=1$ (normalization).
\smallskip

\noindent This correspondence is given by (compare \eqref{eq4.D} with \eqref{coherenceGL}):
$$
\vp(\la)=\Phi\left(\frac{Q_\la(;t)}{(1-t)^{|\la|}}\right), \quad \la\in\Y.
$$
\end{remark}

We describe the boundary of $\Y^\HL(t)$ in the next proposition. To state it, we need a little preparation.  
Let $\R_{\ge0}$ be the set of nonnegative real numbers and $\R_{\ge0}^\infty$ be the direct product of countably many copies of $\R_{\ge0}$, equipped with the product topology.

\begin{definition}\label{def4.A}
Let $\Om(t)$ be the set of pairs $\om=(\al,\be)\in\R_{\ge0}^\infty\times\R_{\ge0}^\infty$ such that 
$$
\al_1 \geq \al_2 \geq \dots \geq 0, \quad \be_1 \geq \be_2 \geq \dots \geq 0,\quad 
 \sum_{i = 1}^{\infty}{\al_i} + (1-t)^{-1}\sum_{i = 1}^{\infty}{\be_i} \le1.
$$
Note that $\Om(t)$ is a compact set. 
\end{definition}

Let $C(\Om(t))$ be the algebra of continuous functions on $\Om(t)$ with pointwise operations. We are going to define an algebra morphism $\Sym\to C(\Om(t))$. Since $\Sym$ is freely generated by the power-sums $p_1, p_2, \cdots$, it suffices to specify their images --- the functions $p_k(\om)$. We set
\begin{equation}\label{eq4.A}
p_1(\om)\equiv1, \qquad p_k(\om)=\sum_{i=1}^\infty\al_i^k +(-1)^{k-1}(1-t^k)^{-1}\sum_{i=1}^\infty \be_i^k, \quad k\ge2.
\end{equation}
It is readily checked that the functions $p_k(\om)$ are continuous. In this way we turn any element $F\in\Sym$ into a continuous function $F(\om)$ on $\Om(t)$. In particular, the HL functions $Q_\la(;t)\in\Sym$ are turned into continuous functions on $\Om$, which will be denoted by $Q_\la(\om;t)$. 
Let us call \eqref{eq4.A} the \emph{$\om$-specialization of the algebra $\Sym$}.

\begin{remark}[cf. Bufetov--Petrov \cite{BP}, Remark 2.6]\label{rem4.BP}
Let us associate with each $\om=(\al,\be)\in\Om(t)$ the triple $(\al,\wt\be,\wt\ga)$, where
$$
\wt\ga:=1-\sum_i \al_i-(1-t)^{-1}\sum_i\be_i\ge0
$$
and $\wt\be$ is the following double infinite collection of parameters: 
$$
\wt\be:=\{\be_{ij}: i,j=1,2,\dots\}, \quad \be_{ij}:=\be_it^{j-1}
$$
(cf. Macdonald \cite[Ch. III, \S2, Example 7]{Mac}). 
In this notation, \eqref{eq4.A} can be rewritten as 
\begin{equation}\label{eq4.B}
p_1(\om)=\sum_{i=1}^\infty\al_i+\sum_{i,j=1}^\infty\be_{ij}+\wt\ga; 
\qquad p_k(\om)=\sum_{i=1}^\infty \al_i^k+(-1)^{k-1}\sum_{i,j=1}^\infty \be_{ij}^k, \qquad k\ge2.
\end{equation}

The formal specialization \eqref{eq4.B} turns any symmetric function $F\in\Sym$ into an \emph{extended symmetric function} in the variables $(\al,\wt\be,\wt\ga)$, in the terminology of Vershik--Kerov \cite[Sect. 6]{VK-1990}.  
\end{remark}

\begin{proposition}[Matveev]\label{prop4.B}
Recall that $t$ is a fixed number in $(0,1)$ and $\Y^\HL(t)$ is the HL-deformed Young graph with parameter $t$. The points in the boundary of\/ $\Y^\HL(t)$ are parametrized by the elements $\om\in\Om(t)$. Specifically, given $\om\in\Om(t)$, the corresponding extreme harmonic function $\vp_\om\in\ex(\Harm_1(\Y^\HL(t)))$ is given by 
$$
\vp_\om(\la):=\frac{Q_\la(\om;t)}{(1-t)^{|\la|}}, \quad \la\in\Y.
$$
\end{proposition}

This is a special case of Proposition 1.6 in Matveev \cite{Mat}. 

\begin{corollary}\label{cor4.A}
The cone $\Harm_+(\Y^\HL(t))$ is isomorphic to the cone of finite Borel measures on the compact space $\Om(t)$. Specifically, the nonnegative harmonic functions on $\Y^\HL(t)$ are precisely the functions of the form
$$
\vp(\la)=\frac1{(1-t)^{|\la|}}\int_{\om\in\Om(t)} Q_\la(\om;t) \boldsymbol{m}(d\om), \qquad \la\in\Y,
$$
where $\boldsymbol{m}$ is a finite Borel measure on $\Om(t)$. 
\end{corollary}

\begin{proof}
This follows from Proposition \ref{prop4.B} combined with Proposition \ref{prop4.A}. A subtle point: we must also make sure that the bijection $\Om(t)\to \ex(\Harm_1(\Y^\HL(t)))$ established in Proposition \ref{prop4.B} is a Borel isomorphism. But this follows from the fact that this map is continuous and hence a homeomorphism, because $\Om(t)$ is compact.
\end{proof}

\subsection{The branching graph $\Ga^{GL(\infty, q)}$}\label{harmGL}

To each pair of diagrams $\mu\in\Y_n$ and $\la\in\Y_{n+1}$, we assign a number $\Li^{n+1}_n(\la, \mu)$ as follows. Pick a nilpotent matrix $X\in\Nil(\gl(n, q))$ of Jordan type $\mu$ and consider the augmented matrices $Y\in\gl(n+1,q)$ of the form 
\begin{equation}\label{eq4.F}
Y := \begin{bmatrix} X & x\\ 0 & 0 \end{bmatrix},\quad x\in\F^n.
\end{equation}
Any such matrix $Y$ is nilpotent. By definition, $\Li^{n+1}_n(\la, \mu)$ is the number of those $Y$'s that have Jordan type $\la$. The definition is correct because this number does not depend on the choice of $X\in\{\mu\}$. 

The next proposition provides an explicit formula for $\Li^{n+1}_n(\la, \mu)$. Its proof was given in Kirillov \cite[Sect. 2.3]{K} and Borodin  \cite[Theorem 2.3]{B}; we present it in Section \ref{sect9}.

\begin{proposition}\label{prop_borodin}
Let $n\in\Z_{\geq 0}$, $\mu\in\Y_n$, $\la\in\Y_{n+1}$.

{\rm(i)} $\Li^{n+1}_n(\la, \mu) = 0$ unless $\mu\nearrow\la$.

{\rm(ii)} Suppose $\mu\nearrow\la$ and denote by $k$ the column number of the single box in $\la\setminus\mu$. Then
\begin{equation}\label{eq4.C}
\Li^{n+1}_n(\la, \mu) = \begin{cases} q^{n - \sum_{j \geq k}{m_j(\mu)}}(1 - q^{-m_{k-1}(\mu)}),&\text{ if }k > 1,\\  q^{n-\sum_{j \geq 1}{m_j(\mu)}},&\text{ if } k = 1. 
\end{cases}
\end{equation}
\end{proposition}

Recall that the notation $m_i(\mu)$ is explained in \eqref{eq2.A}.  Note that the first expression is also applicable in the case $k=1$ if we agree that $m_0(\mu)=+\infty$ and so $q^{-m_0(\mu)} = 0$.

An important remark is that $\Li^{n+1}_n(\la, \mu)>0$ whenever $\mu\nearrow\la$. 

In the next lemma we connect $\Li^{n+1}_n(\la, \mu)$ with the coefficients in the Pieri rule \eqref{pieri_HL}. 

\begin{lemma}[cf. \cite{GKV}, p. 369]\label{lemma4.A}
For any Young diagrams $\mu\nearrow\la$  we have 
$$
\Li^{n+1}_n(\la, \mu)=\psi_{\la/\mu}(q^{-1})\cdot\dfrac{q^{n(\mu)-n(n-1)/2}}{q^{n(\la)-(n+1)n/2}}, \quad n:=|\mu|.
$$
\end{lemma}

Recall that the quantity $n(\ccdot)$ is defined in \eqref{eq2.A}. 

\begin{proof}
Set $n:=|\mu|$. From the comparison of \eqref{eq4.C} with \eqref{psi_def} it follows that the desired relation is equivalent to the equality
$$
n - \sum_{j \geq k}{m_j(\mu)} \stackrel{?}{=} \left(n(\mu)-\frac{n(n-1)}2\right)-\left(n(\la)-\frac{(n+1)n}2\right),
$$
where $k$ is the column number of the unique box in $\la\setminus\mu$. After simplification, the equality is equivalent to
$$
n(\la)-n(\mu) \stackrel{?}{=} \sum_{j \geq k}{m_j(\mu)}.
$$
But this equality follows from the definition of the quantity $n(\la)$.
\end{proof} 

\begin{definition}
Let $\Ga^{GL(\infty,q)}$ be the branching graph whose vertex and edge sets are the same as in the Young graph $\Y$, and  the weights assigned to the edges $\mu\nearrow\la$ are the quantities  $\Li^{n+1}_n(\la, \mu)$, $n=|\mu|$.
\end{definition}

\begin{proposition}\label{HLgraph}
The nonnegative harmonic functions on the graph $\Ga^{GL(\infty, q)}$ are precisely the functions of the form
\begin{equation} 
\vp^{GL(\infty,q)}(\la)=\vp(\la)\cdot q^{n(\la)-\binom{|\la|}2}, \qquad \la\in\Y, 
\end{equation}
where $\vp$ is an arbitrary nonnegative harmonic function on the graph $\Y^\HL(q^{-1})$. 
\end{proposition}

\begin{proof}
Lemma \ref{lemma4.A} shows that the graphs $\Y^\HL(q^{-1})$ and $\Ga^{GL(\infty,q)}$ are similar in the sense of Definition \ref{def4.sim}, with the gauge function from $\Y^\HL(q^{-1})$ to  $\Ga^{GL(\infty,q)}$ equal to 
\begin{equation}\label{eq4.f}
f(\la)=q^{n(\la)-\binom{|\la|}2}, \quad \la\in\Y.
\end{equation}
Then we apply Lemma \ref{homol_graphs}.
\end{proof} 

\subsection{Final result: description of measures}
Any measure $P\in\M^\glinfty_0$ is uniquely determined by its values on the elementary cylinder sets $\Cyl_n(X)$, where $X\in\gl(n,q)$ and $n\in\Z_{\ge0}$. Note that $P(\Cyl_n(X))=0$ unless $X$ is nilpotent.

\begin{lemma}\label{lemma4.B}
There is a bijective correspondence  between measures $P\in\M^\glinfty_0$ and nonnegative harmonic functions $\vp^\glinfty$ on the branching graph $\Ga^\glinfty$, uniquely characterized by the property that for any $n\in\Z_{\ge0}$ and any nilpotent matrix $X\in\gl(n,q)$ of a given Jordan type $\mu\in\Y_n$, one has
\begin{equation}\label{eq4.E}
P(\Cyl_n(X))=\vp^\glinfty(\mu).
\end{equation}
\end{lemma}

\begin{proof}
Let $P\in\M^\glinfty_0$. Because $P$ is invariant, the left-hand side of \eqref{eq4.E} only depends on the Jordan type $\mu$ of $X$. The elementary cylinder $\Cyl_n(X)$ can also be viewed as a cylinder set of level $n+1$. As such, it splits into a disjoint union of elementary cylinders $\Cyl_{n+1}(Y)$, where $Y$ ranges over the set of matrices of the form \eqref{eq4.F}. This entails the equality
$$
P(\Cyl_n(X))=\sum_Y P(\Cyl_{n+1}(Y)).
$$
By the definition of the quantities $\Li^{n+1}_n(\la,\mu)$, among the $Y$'s in the sum, exactly $\Li^{n+1}_n(\la,\mu)$ of them are of Jordan type $\la$, for each $\la\in\Y_{n+1}$. Therefore, taking \eqref{eq4.E} as the definition of a function $\vp^\glinfty$, we see that this function is a nonnegative harmonic function on the graph $\Ga^\glinfty$. 

Conversely, let $\vp^\glinfty$ be a nonnegative harmonic function on the graph $\Ga^\glinfty$. Then the above argument shows that formula $\eqref{eq4.E}$ gives rise to a nonnegative, finitely-additive set function $P$ on cylinder sets contained in the set of pronilpotent matrices. By virtue of Lemma \ref{lemma3.A}, $P$ extends to a true Radon measure on $\La(q)$, which is supported on the set of pronilpotent matrices. Its invariance is evident from the very construction.

This completes the proof. 
\end{proof}

Combining Lemma \ref{lemma4.B}, Proposition \ref{HLgraph}, and  Corollary \ref{cor4.A}, we finally obtain the following description of the measures $P\in\M_0^\glinfty$. 

\begin{theorem}\label{thm4.A}
{\rm(1)} The measures $P\in\M_0^\glinfty$ are in one-to-one correspondence with the finite Borel measures $\boldsymbol{m}$ on the compact space $\Om(q^{-1})$.

\smallskip

{\rm(2)} Under this correspondence $P\leftrightarrow \boldsymbol{m}$, the mass of an elementary cylinder set $\Cyl_n(X)$, where $X\in\Nil(\gl(n,q))$ is a nilpotent matrix of Jordan type $\la\in\Y_n$, is given by 
$$
P(\Cyl_n(X))=\frac{q^{n(\la)-n(n-1)/2}}{(1-q^{-1})^n}\int_{\om\in\Om(q^{-1})} Q_\la(\om;q^{-1})\boldsymbol{m}(d\om).
$$ 

\smallskip

{\rm(3)} The mass $P(\La_0(q))$ is equal to the total mass of $\boldsymbol{m}$. Thus, the normalized measures $P$ {\rm(}i.e. the ones with  $P(\La_0(q))=1${\rm)} correspond to probability measures  $\boldsymbol{m}$.  

\smallskip

{\rm(4)} Normalized ergodic measures $P$ correspond to Dirac measures $\m$ and hence are parametrized by points $\om\in\Om(q^{-1})$. For such measures $P=P_\om$,  the above formula reduces to
\begin{equation}\label{eq4.I}
P_\om(\Cyl_n(X))=\frac{q^{n(\la)-n(n-1)/2}}{(1-q^{-1})^n} Q_\la(\om;q^{-1}), \qquad X\in\{\la\}, \quad \la\in\Y_n.
\end{equation}
\end{theorem}

The last formula coincides with formula (1.3) in Bufetov--Petrov \cite{BP}, and with (4.3) in Gorin--Kerov--Vershik \cite{GKV}.

\subsection{Examples and remarks}
In this subsection, we use the standard notation
$$
(u;t)_n:=\prod_{j=1}^n(1-ut^{j-1}), \quad  (u;t)_\infty:=\prod_{j=1}^\infty (1-u t^{j-1}).
$$
We also employ the notation from Definition \ref{def4.A} and Remark \ref{rem4.BP}.  

We give some examples of normalized ergodic measures $P_\om$, which correspond to certain particular values of the parameters $\om=(\al,\be)\in\Om(q^{-1})$. We also write down the complementary parameter $\wt\ga=1-\sum \al_i-\sum (1-q^{-1})^{-1}\be_i$.  To simplify the notation throughout this subsection, we set $t:=q^{-1}$. 

\begin{example}[$\al\equiv0$, $\be=(1-t,0,0,\dots)$, $\wt\ga=0$]\label{ex5.A}
The corresponding normalized ergodic measure $P_\om$ is the Dirac measure at the point $0\in\La(q)$. This is the simplest (and trivial) example. The restriction of $P$ to an elementary cylinder $\Cyl_n(X)$ vanishes unless $X=0_n$, which is the unique matrix of type $(1^n)$.
\end{example}

\begin{example}[$\al=(1,0,0,\dots)$, $\be\equiv0$, $\wt\ga=0$]
For the corresponding normalized ergodic measure $P$, its restriction to an elementary cylinder $\Cyl_n(X)$ vanishes unless $X$ has type $\la=(n)$. The restriction of $P_\om$ to $\La_0(q)$ admits a nice description: in the matrix coordinates $m_{ij}$, $i<j$, this is a product measure, such that each coordinate $m_{i,i+1}$ does not vanish and is uniformly distributed  on $\F\setminus\{0\}$, while each coordinate $m_{ij}$ with $j\ge i+2$ is uniformly distributed on $\F$.  
\end{example}

\begin{example}[$\al=(1-t, (1-t)t, (1-t)t^2,\dots)$, $\be\equiv0$, $\wt\ga=0$]\label{ex4.Haar}
The normalized ergodic measure with these parameters will be denoted by $P^\Haar$. For this measure, formula \eqref{eq4.I} takes the simple form
$$
P^\Haar(\Cyl_n(X))=t^{n(n-1)/2}, \qquad X\in\{\la\}, \quad \la\in\Y_n,
$$
because $Q_\la(1,t,t^2,\dots; t)=t^{n(\la)}$ (see \cite[Ch. III, \S2, Ex. 1]{Mac}). From this expression, it is seen that $P^\Haar$ is the   restriction to $\Nil(\La(q))$ of the Haar measure on the additive group $\La(q)$, with the normalization $P^\Haar(\La_0(q))=1$. The Haar measure on $\La_0(q)$ was the subject of Borodin's work \cite{B}. 
\end{example}

\begin{example}[$\al\equiv0$, $\be\equiv0$, $\wt\ga=1$]
The specialization $\Sym\to\R$ with these parameters is the ``Plancherel specialization'' that sends $p_1$ to $1$ and all other $p_k$'s to $0$. For this reason we denote the corresponding normalized ergodic measure by $P^\Planch$. 
Formula \eqref{eq4.I} takes the form
$$
P^\Planch(\Cyl_n(Y))=\frac{t^{\frac{n(n-1)}2-n(\la)}}{n!}\,X^\la_{(1^n)}(t), \qquad Y\in\{\la\}, \quad \la\in\Y_n,
$$
where $X^\la_{(1^n)}(t):=\langle p_1^n,Q_\la(;t)\rangle_t$ and $\langle\ccdot,\ccdot \rangle_t$ is the HL inner product defined in \cite[Ch. III, \S4]{Mac}. It is known (see \cite[Ch. III, \S7]{Mac}) that the $X^\la_{(1^n)}(t)$ are polynomials in $t$, which can be written as certain combinatorial sums (\cite[Ch. III, \S7, Ex. 4]{Mac}). 
\end{example}

\begin{remark}[Sizes of classes $\{\la\}$]
Let $\la\in\Y_n$. One can show that the number of nilpotent matrices in the conjugacy class $\{\la\}\subset\gl(n,q)$ is given by 
\begin{equation}
|\{\la\}|=(t;t)_n\, t^{2n(\la)-n(n-1)}(b_\la(t))^{-1}, \quad t=q^{-1}.
\end{equation}
\end{remark}

\begin{remark}[The mass of $\La_n(q)$]
Set $t=q^{-1}$, fix $\om\in\Om(t)$, and let $P_\om$ be the corresponding normalized ergodic measure, as in \eqref{eq4.I}. The following formula holds 
\begin{equation}
P_\om(\La_n(q))=\frac{(t;t)_n h_n(\om)}{(1-t)^n t^{n(n-1)/2}}.
\end{equation} 
Here $h_n(\om)$, the $\om$-specialization of $h_n\in\Sym$, is found from the generating series 
$$
H(\om)(z):=1+\sum_{n=1}^\infty h_n(\om)z^n=e^{\wt\ga z}\prod_{i=1}^\infty \frac{(-\be_i z;t)_\infty}{1-\al_i z}.
$$
The proof is based on the fundamental Cauchy identity for the HL functions.
\end{remark}

\begin{remark}[Law of large numbers]
Each normalized ergodic measure $P$ gives rise to a sequence $\{\la(n)\}$ of random Young diagrams of growing size $n$. Bufetov and Petrov \cite{BP} proved a law of large numbers for the row and column lengths of $\la(n)$. In the case $P=P^\Haar$, this was done earlier by Borodin \cite{B}, together with a central limit theorem.  
\end{remark}

\section{Description of $\PGL$}\label{sect5}

In this section we establish an affine isomorphism between the cone $\PGL$ and a direct product of countably many copies of the cone $\PGL_0$ (Propositions \ref{prop5.A} and \ref{prop5.B}). Because the structure of $\PGL_0$ was described in Section \ref{sect4}, we obtain in this way a complete description of invariant Radon measures on $\La(q)$.

\subsection{A partition of $\La(q)$}

As explained in Subsection \ref{sect2.4}, the type of each matrix $X\in\gl(n,q)$ (i.e. the full invariant of its conjugacy class) is represented  by a pair $\tau=(\si,\la)\in\T_n$, where $\si\in\Si_m$ is a nonsingular type and $\la\in\Y_{n-m}$ is a Young diagram, for some $m\le n$. 

We are going to partition $\La(q)$ into countably many invariant subsets $\La^\si(q)$, indexed by elements of the set 
$$
\Si:=\bigsqcup_{m=0}^\infty\Si_m.
$$
It is convenient to set $|\si|:=m$ if $\si\in\Si_m$, and call this number the \emph{size} of $\si$.  

Recall that for an infinite matrix $M\in\La(q)$, we denote its upper-left $n\times n$ corner by $M^{\{n\}}$. 

\begin{lemma}\label{lemma5.A}
Fix $M\in\La(q)$. For each $n$, let $(\si_n,\la^{(n)})\in\T_n$ stand for the type of the corner $M^{\{n\}}\in\gl(n, q)$. Let $n_0$ be the minimal number such that $M\in\La_{n_0}(q)$. We have
$$
\si_n=\si_{n+1}, \quad \la^{(n)}\nearrow \la^{(n+1)}, \qquad n\ge n_0.
$$
\end{lemma}

\begin{proof}
Suppose $n\ge n_0$ and let $X := M^{\{n\}}\in\gl(n, q)$. Then the corner $M^{\{n+1\}}$ is of the form
$$
Y = \begin{bmatrix} X & x \\ 0 & 0 \end{bmatrix},\quad x\in\F^n.
$$
Conjugating $X$ by an appropriate matrix from the subgroup $GL(n,q)\subset GL(n+1,q)$, we may assume that $X$ has the block form $\begin{bmatrix} S & 0\\ 0 & N\end{bmatrix}$, where $S$ is nonsingular (of type $\si_n$) and $N$ is nilpotent (of Jordan type $\la^{(n)}$). Thus, $Y$ can be represented by the $3\times 3$ block matrix
$$
Y=\begin{bmatrix}
S & 0 & x_1\\ 0 & N & x_2 \\ 0 & 0 & 0
\end{bmatrix} 
$$
Next,  we can kill $x_1$ by conjugating $Y$ with an appropriate unitriangular matrix. Indeed,
$$
\begin{bmatrix}
1 & 0 & z\\ 0 & 1 & 0 \\ 0 & 0 & 1
\end{bmatrix} 
\begin{bmatrix}
S & 0 & x_1\\ 0 & N & x_2 \\ 0 & 0 & 0
\end{bmatrix} 
\begin{bmatrix}
1 & 0 & -z\\ 0 & 1 & 0 \\ 0 & 0 & 1
\end{bmatrix} 
=
\begin{bmatrix}
S & 0 & x_1-Sz\\ 0 & N & x_2 \\ 0 & 0 & 0
\end{bmatrix},
$$
and since $S$ is nonsingular, there exists $z$ such that $Sz=x_1$. 

Now the desired result follows from the first claim in Proposition \ref{prop_borodin}.
\end{proof}

The lemma shows that for any $M\in\La(q)$, the nonsingular type of the corner $M^{\{n\}}$ stabilizes; let us call it the \emph{stable nonsingular type of $M$} and denote it by $\si(M)$.  The correspondence $M\mapsto \si(M)$ gives rise to a map $\La(q)\to\Si$. We denote its fibres by $\La^\si(q)$:
$$
\La^\si(q):=\{M\in\La(q): \si(M)=\si\}, \qquad \si\in\Si.
$$
Thus, we obtain a partition
$$
\La(q)=\bigsqcup_{\si\in\Si}\La^\si(q).
$$

Note that the set $\Nil(\La(q))$ of pronilpotent matrices is one of the parts --- it corresponds to the only element of $\Si_0$, the empty nonsingular type.  

\begin{lemma}\label{lemma5.B}
Each $\La^\si(q)$, $\si\in\Si$, is a nonempty, $GL(\infty,q)$-invariant, clopen subset.
\end{lemma}

\begin{proof}
Fix an arbitrary $\si\in\Si$. From Lemma \ref{lemma5.A} it is seen that for any $n\ge |\si|$, the intersection $\La^\si(q)$ with $\La_n(q)$ is a nonempty cylinder set. It follows that $\La^\si(q)$, $\si\in\Si$, is nonempty, $GL(\infty,q)$-invariant, and open. Finally, it is also closed, because its complement is open.  
\end{proof}

\subsection{The structure of $\PGL$}\label{sect5.2}

\begin{definition}
For $\si\in\Si$, let $\PGL_\si \subset\PGL$ be the subset of measures which are supported on the clopen subset $\La^\si(q)$.
\end{definition}

Given $P\in\PGL$ and $\si\in\Si$, we denote by $P_\si$ the restriction of $P$ to $\La^\si(q)$. It is an element of $\M_\si^{GL(\infty, q)}$ and we can write
 $$
P = \sum_{\si\in\Si}P_\si, \qquad P_\si\in\PGL_\si.
$$

Observe that $\M_\si^{GL(\infty, q)}$ is a convex cone, for any $\si\in\Si$.

Finally, observe that if $\si_0$ is the only element of $\Si_0$, then $\PGL_{\si_0}$ coincides with the convex cone $\PGL_0$ of measures supported on the set of pronilpotent matrices, that was examined in Section \ref{sect4}.

\begin{proposition}\label{prop5.A}
The above decomposition determines an affine-isomorphism $\PGL\cong \prod_{\si\in\Si}{\PGL_\si}$.
\end{proposition}

\begin{proof}
We only need to check that for any choice of measures $P_\si\in\PGL_\si$, $\si\in\Si$, their sum is a Radon measure. To see this recall that any compact subset of $\La(q)$ is contained in $\La_n(q)$ for $n$ large enough. On the other hand, $\La_n(q)$ intersects only those sets $\La^\si(q)$ for which $|\si|\le n$, and there are only finitely many such $\si$'s. This completes the proof.
\end{proof}

\begin{proposition}\label{prop5.B}
For any $\si\in\Si$, there is an affine-isomorphism $\PGL_0 \stackrel{\cong}{\to}\PGL_\si$ of convex cones.
\end{proposition} 
\begin{proof}
Fix $\si\in\Si$ and let $s:=|\si|$. We may assume that $s>0$, because if $s=0$, then $\si$ is the unique element of $\Si_0$ and $\PGL_\si = \PGL_0$. Let $\mu\in\Y_n$ and $\la\in\Y_{n+1}$. Pick a matrix 
$$
X=\begin{bmatrix} S & 0 \\ 0 & N
\end{bmatrix} \in\gl(s+n,q),
$$
where $S\in\gl(s,q)$ is a nonsingular matrix of type $\si$ and $N\in\gl(n,q)$ is a nilpotent matrix of Jordan type $\mu\in\Y_n$. 

Denote by $L^{n+1}_n(\la,\mu\mid\si)$ the number of column vectors $x\in\F^{s+n}$ such that the matrix
$$
Y:=\begin{bmatrix}
X & x \\ 0 & 0
\end{bmatrix}
=\begin{bmatrix}
S & 0 & x_1 \\ 0 & N & x_2 \\ 0 & 0 & 0
\end{bmatrix}  
\in\gl(s+n+1,q)
$$
has type $(\si,\la)$. This number does not depend on the choice of $X$. 

From the proof of Lemma \ref{lemma5.A} we see that 
\begin{equation}\label{eq5.A}
L^{n+1}_n(\la,\mu\mid\si)=q^s L^{n+1}_n(\la,\mu),
\end{equation}
where the quantity $L^{n+1}_n(\la,\mu)$ is defined in Subsection \ref{harmGL}. Recall that $L^{n+1}_n(\la,\mu)$ is nonzero precisely when $\mu\nearrow \la$. 

Let $\Ga^{GL(\infty,q)}_\si$ denote the branching graph with the same vertices and edges as in the Young graph, and with the  edge weights $L^{n+1}_n(\la,\mu\mid\si)$. The same argument as in Lemma \ref{lemma4.B} yields an isomorphism between the cone $\PGL_\si$ and the cone of nonnegative harmonic functions on the graph $\Ga^{GL(\infty,q)}_\si$. 

On the other hand, from the relation \eqref{eq5.A}, it is seen that the graph $\Ga^{GL(\infty,q)}_\si$ is similar (in the sense of Definition \ref{def4.sim}) to the graph $\Ga^{GL(\infty,q)}$. We know that nonnegative harmonic functions on $\Ga^{GL(\infty,q)}$ correspond to measures $P\in\PGL_0$ (Lemma \ref{lemma4.B}). This yields the desired isomorphism.  
\end{proof}

\section{Conjugacy classes in skew-Hermitian matrices}\label{U_def}

\subsection{The unitary group over a finite field}

Denote $\FF := \F_{q^2}\supset\F$, and let $F : \FF\to\FF$ be the \emph{Frobenius map} defined by $F(x) := x^{q}$. It is the non-trivial involutive automorphism of $\FF$ that fixes $\F$. We denote $\overline{x} := F(x)$. In what follows we will assume that $q$ is not a power of $2$ (the case of characteristic $2$ requires only minor modifications, but to simplify the exposition we exclude it).
Fix $\eps\in\FF\setminus\F$ such that $\overline{\eps} = -\eps$.
The particular choice of $\eps$ is unimportant, as we only will use that $\overline{a + b\eps} = a - b\eps$, for any $a, b\in\F$.
For finite fields, the Frobenius map $x \mapsto \overline{x}$ plays the role that conjugation plays for $\C$, and $\eps$ plays the role of the imaginary unit $i$.
In the same spirit, the \emph{conjugate transpose} $A^*$ of a rectangular matrix $A$ with entries in $\FF$ is obtained by transposing $A$ and applying the Frobenius map to each entry.

Let $E$ be a finite-dimensional space over $\FF$; by a \emph{sesquilinear form} on $E$ we mean a map $\tau : E\times E\to\FF$ which is additive in each variable and satisfies
$$
\tau(ax, y) = a\tau(x, y),\quad \tau(x, ay) = \overline{a}\tau(x, y),\qquad x, y\in E,\quad a\in\FF.
$$
The sesquilinear form $\tau$ is \emph{Hermitian} if
$$
\tau(x, y) = \overline{\tau(y, x)},\qquad x, y\in E.
$$
All nondegenerate sesquilinear Hermitian forms on a fixed vector space are equivalent, see e.g. \cite[Theorem 4.1]{BC0} or \cite[Sect. 4]{Lewis}.
A linear operator $A: E\to E$ is said to be \emph{$\tau$-skew-Hermitian} if 
\begin{equation}\label{tau_sH}
\tau(Au, v) + \tau(u, Av) = 0, \qquad u, v\in E,
\end{equation}
whereas $A$ is said to be \emph{$\tau$-Hermitian} if
\begin{equation}\label{tau_H}
\tau(Au, v) = \tau(u, Av), \qquad u, v\in E.
\end{equation}

To denote the dimension of vector spaces and the size of matrices, we will use the symbol $N$, because later we will need to distinguish between the cases of even $N=2n$ and odd $N=2n+1$ dimensions.

If $E = \FF^N$ is the vector space of $N$-column vectors, then a nondegenerate sesquilinear form $\tau$ on $E$ corresponds to a nonsingular $N\times N$ matrix $T$ via the relation:
$$
\tau(x, y) = y^*Tx.
$$
The form $\tau$ is Hermitian iff $T^* = T$.

On the coordinate space $\FF^N$, let us fix the nondegenerate sesquilinear Hermitian form $\tau : \FF^N\times\FF^N \to \FF$ corresponding to the matrix $W_N$ with $1$'s in the secondary diagonal and $0$'s elsewhere:
\begin{equation}\label{matrixW}
W_N := \begin{bmatrix}  & &  &  & 1\\ & & & 1 & \\ & &\ddots& & \\ & 1 & & & \\ 1 & & & & \end{bmatrix}.
\end{equation}
In other words, if $e_1, \dots, e_N$ are the vectors from the canonical basis of $\FF^N$, then the form $\tau$ is given by
\begin{equation}\label{std_form}
\tau(e_r, e_s) = \delta_{r, N+1-s}, \qquad r,s = 1, \dots, N.
\end{equation}
Given a matrix $A\in\Mat_N(\FF)$, we say that $A$ is \emph{$\tau$-skew-Hermitian}, or \emph{$\tau$-Hermitian}, if the same is true of the corresponding linear operator $A : \FF^N\to\FF^N$ (see \eqref{tau_sH} and \eqref{tau_H}).
More directly, if we define the \emph{$\tau$-conjugate} of $A$ by $A^{\sharp} := W_NA^*W_N$, then $A$ is $\tau$-skew-Hermitian if $A^{\sharp} = -A$, and $A$ is $\tau$-Hermitian if $A^{\sharp} = A$.
For simplicity, we call a matrix $A\in\Mat_N(\FF)$ simply skew-Hermitian if $A^{\sharp} = -A$, or Hermitian if $A^{\sharp} = A$.

The \emph{unitary group} $U(N, q^2)$ is the subgroup of matrices in $GL(N, q^2) = GL(N, \FF)$ that preserve $\tau$:
\begin{equation*}
U(N, q^2) := \{ g\in GL(N, q^2) \mid \tau(gu, gv) = \tau(u, v), \text{ for all } u, v\in\FF^N \}.
\end{equation*}

The Lie algebra of $U(N, q^2)$, to be denoted $\uu(N, q^2)$, is the Lie subalgebra of $\gl(N, q^2) = \gl(N, \FF)$ consisting of the skew-Hermitian matrices:
\begin{equation*}
\uu(N, q^2) := \{ A\in\gl(N, q^2) \mid \tau(Au, v) + \tau(u, Av) = 0, \text{ for all } u, v\in\FF^N \}.
\end{equation*}
The group $U(N, q^2)$ acts on $\uu(N, q^2)$ by matrix conjugation. As in the case of $GL(N,q)$, this adjoint action may be identified with the coadjoint action.

Denote the set of Hermitian matrices by
$$\Herm(N, q^2) := \{ A\in\gl(N, q^2) \mid \tau(Au, v) = \tau(u, Av), \text{ for all }u, v\in\FF^N \}.$$
It is also a $U(N, q^2)$-module, with action given by conjugation. Since $\Herm(N, q^2) = \eps\cdot\uu(N, q^2)$, the $U(N, q^2)$-modules $\Herm(N, q^2)$ and $\uu(N, q^2)$ are equivalent. We can switch from one space to the other if necessary. 

\subsection{Nilpotent conjugacy classes in $\uu(N, q^2)$}\label{unit_orbs}

By a conjugacy class in $\uu(n,q^2)$ we mean a $U(N,q^2)$-orbit. Let $\Nil(\uu(N, q^2))\subseteq\uu(N, q^2)$ denote the subset of nilpotent matrices. It is $U(N,q^2)$-invariant, so we may speak of nilpotent conjugacy classes in $\uu(N,q^2)$ just as we did for $\gl(n,q)$. We need the following result.

\begin{proposition}\label{prop6.A}
Nilpotent conjugacy classes in $\uu(N,q^2)$ are parametrized by partitions $\la\in\Y_N$ ---  exactly as in the case of\/ $\gl(N,q^2)$.
\end{proposition}

\begin{proof}

The Jordan type of a nilpotent matrix is obviously an invariant of its conjugacy class. Switching to $\tau$-Hermitian matrices, we have to check two claims:

\smallskip

(1)  for each $\la\in\Y_N$, there exists a matrix $X\in\Herm(N, q^2)$ with this Jordan type;

(2) if two matrices  $X,Y\in\Herm(N, q^2))$ are of the same Jordan type, then they are conjugated by an element of the subgroup $U(N,q^2)\subseteq GL(N,q^2)$.
\smallskip

To show (1) we exhibit a concrete model for $X$.

Consider the $N$-dimensional vector space $V$ over $\FF$ with the distinguished basis $\{v_{ij}\}$ indexed by the boxes $(i,j)$ of the Young diagram $\la$.
Let $X: V\to V$ be the operator defined by $X v_{ij}=v_{i,j-1}$, with the understanding that $v_{i, 0}:=0$.
This operator is nilpotent and of Jordan type $\la$. Next, let $\tau$ be the sesquilinear form on $V$ with the property that $\tau(v_{ij}, v_{i, \la_i-j+1})=1$ and all other scalar products are $0$.
Clearly, $\tau$ is nondegenerate and Hermitian, and the operator $X$ is $\tau$-Hermitian. 
Because all nondegenerate sesquilinear Hermitian forms of dimension $N$ are equivalent, this gives the desired result.

To show (2), we use the fact that if two matrices $X,Y\in\Herm(N,q^2)$ are conjugated by an element of $GL(N,q^2)$, then they are also conjugated by an element of $U(N,q^2)$. This follows from the Lang--Steinberg theorem, see the expository paper by Springer and Steinberg in \cite{SAG}, Section E, Example 3.5 (a) (this example concerns conjugacy classes in the group $U(N,q^2)$, but the same argument works for its Lie algebra $\uu(N,q^2)$ or the space $\Herm(N,q^2)$).
\end{proof}

\begin{proposition}\label{prop6.B}
Each nilpotent conjugacy class in $\uu(N,q^2)$ contains a strictly upper triangular matrix. 
\end{proposition}

Note that this is not evident, in contrast to the case of $\gl(n,q)$.
This statement depends on our chosen presentation of $\uu(N, q^2)$, i.e. on the choice of the form $\tau$ used to define it.

\begin{proof}
In invariant terms, the claim is equivalent to the following. Let $V$ be a vector space over $\wt\F$, of dimension $2m$ or $2m+1$, $\tau$ be a nondegenerate sesquilinear form on $V$, and $X:V\to V$ be a nilpotent, $\tau$-skew-Hermitian operator. Then there exists a complete flag $\{V_i\}$, which is preserved by $X$ and has the form
\begin{gather*}
\{0\}\subset V_1\subset\dots\subset V_m\subset  V_{m-1}^\perp\subset\dots\subset V_1^\perp\subset V, \quad \text{for $\dim V=2m$,}\\
\{0\}\subset V_1\subset\dots\subset V_m\subset V_m^\perp\subset  V_{m-1}^\perp\subset\dots\subset V_1^\perp\subset V, \quad \text{for $\dim V=2m+1,$}
\end{gather*}  
where the first $m$ subspaces are $\tau$-isotropic and the symbol $(\cdots)^\perp$ means orthogonal complement. 

We prove this by induction on $\dim V$. If $X=0$, there is nothing to prove, so we assume $X\ne0$. Since $X$ is nilpotent, its kernel $\ker X$ is nonzero. We claim that $\ker X$ contains a nonzero isotropic vector $v$ unless $\dim V=1$. Indeed, if $\dim(\ker X)\ge2$, then this holds true because any subspace of dimension $\ge2$ contains a nonzero isotropic vector. Next, if $\dim(\ker X)=1$ and $\ker X$ is non-isotropic, then we have the orthogonal decomposition $V=\ker X\oplus (\ker X)^\perp$, which is also $X$-invariant, but then $(\ker X)^\perp$ must be null and hence $\dim V=1$, because otherwise $X$ would have a nontrivial kernel in $(\ker X)^\perp$, which is impossible. 

The case $\dim V=1$ being trivial, we assume $\dim V\ge2$. Then we take as $V_1$ the one-dimensional isotropic subspace spanned by $v$. If $\dim V=2$, we are done. If $\dim V>2$, then the subspace $V_1^\perp$ is strictly larger than $V_1$. Since it is $X$-invariant (here we use the fact that $X$ is $\tau$-skew-Hermitian), our task is reduced to the quotient space  $V_1^\perp/V_1$. This argument yields the desired induction step. 
\end{proof}

\subsection{General conjugacy classes in $\uu(N, q^2)$}\label{sect6.3}
The arguments of Subsection \ref{sect2.4} are extended to the case of $\uu(N,q^2)$ with minor modifications. 

We denote by $\wt\T_N$ the set of conjugacy classes in $\uu(N, q^2)$. A matrix $X\in\uu(N, q^2)$ belonging to a class $\tau\in\wt\T_N$ is said to be \emph{of type $\tau$}. 

Let $\NSin(\uu(N,q^2))\subset\uu(N,q^2)$ be the subset of nonsingular matrices. It is $U(N,q^2)$-invariant. Let  $\wt\Si_N\subset\wt\T_N$ be the subset of conjugacy classes contained in $\NSin(\uu(N,q^2))$. Elements of $\wt\Si_N$ will be called \emph{nonsingular classes} or \emph{nonsingular types}.

\begin{lemma}\label{lemma6.A}
There is a natural bijection 
$$
\wt\T_N\leftrightarrow \bigsqcup_{s=0}^N (\wt\Si_s\times\Y_{N-s}).
$$
\end{lemma}

(Here we regard $\wt\Si_0$ as a singleton, so that $\wt\Si_0\times\Y_N$ is identified with $\Y_N$.)

\begin{proof}
We argue as in the proof of Lemma \ref{lemma2.A}. In invariant terms, the set $\wt\T_N$ can be identified with the set of equivalence classes of triples $(V,\tau,X)$, where $V$ is an $N$-dimensional vector space over $\wt\F$, $\tau$ is a nondegenerate sesquilinear Hermitian form on $V$, and $X: V\to V$ is a $\tau$-skew-Hermitian operator. We consider again the canonical $X$-invariant decomposition $V=V'\oplus V_0$ with the property that $X\big|_{V'}$ is nonsingular and $X\big|_{V_0}$ is nilpotent, and observe that it is orthogonal with respect to $\tau$. This gives us two invariants: the type of $X\big|_{V'}$ and the Jordan type of $X\big|_{V_0}$, which determine the equivalence class of $(V,\tau,X)$ uniquely. In this way we obtain an embedding $\wt\T_N\hookrightarrow \bigsqcup_{s=0}^N (\wt\Si_s\times\Y_{N-s})$.  Finally, this map is also obviously surjective, which leads to  the desired bijection. 
\end{proof} 

Thus, the parametrization of general conjugacy classes in $\uu(N,q^2)$ is reduced to the explicit description of the nonsingular types. For completeness, we give it in the remark below, but in fact we do not use it  in the present paper.

\begin{remark}[cf. Remark \ref{rem2.A}]\label{rem6.A}
Let $f(x)\in\FF[x]$ be a nonconstant, monic polynomial with nonzero constant coefficient:
\begin{equation}\label{poly_p}
f(x) = x^d + c_1 x^{d-1} + \dots + c_{d-1} x + c_d, \text{ where } d = \deg(f)\geq 1 \text{ and } c_d \neq 0.
\end{equation}
Define $\wt{f}(x) := (-1)^{\deg(f)}\cdot\overline{f(-x)}$, that is,
$$
\wt{f}(x) = x^d + \wt c_1x^{d-1} + \dots + \wt c_{d-1} x + \wt c_d,\qquad \wt c_i := (-1)^i \overline{c_i}.
$$
The roots of the polynomial $\wt{f}(x)$ (in the algebraic closure of $\FF$) are the result of applying the map $x \mapsto -F(x) = -\overline{x} = -x^{q}$ to the roots of $f(x)$.

We say that  $f$ is almost-irreducible if its roots form a single orbit under the map $x \mapsto -F(x)$.  More explicitly, a polynomial  $f$ of the form \eqref{poly_p} is almost-irreducible if either
\smallskip

$\bullet$ $f$ is irreducible in $\FF[x]$ and $f = \wt{f}$ ($\deg(f)$ is odd in this case), or

$\bullet$ $f = g\wt{g}$, where $g$ is irreducible in $\FF[x]$ and $g \neq \wt{g}$ ($\deg(f)$ is even in this case). 

\smallskip

Denote by $\wt\Phi'$ the set of almost-irreducible polynomials in $\FF[x]$. There is a bijective correspondence between elements of $\wt\Si_s$ and maps $\wt{\boldsymbol{\mu}}: \wt\Phi'\to \Y$ such that $\wt{\boldsymbol{\mu}}(f) = \emptyset$ for all but finitely many polynomials $f\in\wt\Phi'$ and
$$
\sum_{f\in\wt\Phi'}\deg(f)|\wt{\boldsymbol{\mu}}(f)|=s,
$$
where $\deg(f)$ is the degree of $f$. This result can be extracted from \cite{Wa} or \cite{E1}. 
\end{remark}

\section{Invariant measures on skew-Hermitian matrices}\label{sect7}

We need to differentiate between unitary groups of even and odd dimension.

It is convenient to index rows and columns of matrices in the even unitary group $U(2n, q^2)$ by the $(2n)$-element set $\{-n, \dots, -1, 1, \dots, n\}$. This leads to the inclusions $U(2n, q^2) \hookrightarrow U(2n+2, q^2)$, and therefore to the inductive limit $U(2\infty, q^2) := \varinjlim U(2n, q^2)$. It is naturally a group of infinite size matrices of format $(\Z\setminus\{0\})\times(\Z\setminus\{0\})$. The group $U(2\infty, q^2)$ is countable, in fact, the $(i, j)$-entry of a matrix from $U(2\infty, q^2)$ equals $\de_{i, j}$, for all but finitely many entries.

Likewise, index rows and columns of matrices in the odd unitary group $U(2n+1, q^2)$ by the $(2n+1)$-element set $\{-n, \dots, 0, \dots, n\}$. This leads to the inductive limit group $U(2\infty+1, q^2) := \varinjlim U(2n+1, q^2)$, which is also countable and whose elements are infinite size matrices of format $\Z\times\Z$.

The superscripts $\e$ and $\o$ make reference to \emph{even} and \emph{odd}, respectively.
Both $U(2\infty, q^2)$ and $U(2\infty+1, q^2)$ are called \emph{infinite unitary groups}.

\subsection{The group $U(2\infty,q^2)$ and the space $\La^{\e}(q^2)$}

The infinite size matrices in this subsection are of format $(\Z\setminus\{0\})\times(\Z\setminus\{0\})$ --- they are now \emph{two-sided infinite} matrices. Let $\Mat^{\e}_{\infty}(q^2)$ be the space of matrices $M = [m_{i, j}]_{i, j\in\Z\setminus\{0\}}$ with entries in $\FF$. For $n\in\Z_{\geq 0}$, let $\La^{\e}_n(q^2)\subset\Mat^{\e}_{\infty}(q^2)$ be the subset of matrices $M$ such that:

\smallskip

$\bullet$  $m_{i, j} = 0$ whenever $i \geq j$, and $\max\{|i|, |j|\} > n$;

$\bullet$ $M$ is skew-Hermitian, i.e. $m_{-b, -a} = -\overline{m_{a, b}} (= -m_{a, b}^q)$ for all $a, b\in\Z\setminus\{0\}$. In particular, $m_{-k, k}\in\eps\cdot\F$ for all $k\in\Z\setminus\{0\}$.

\smallskip

For each $n\in\Z_{\geq 0}$, the set $\La^{\e}_n(q^2)$ is a vector space over $\F$, therefore a commutative additive group.
The set $\La^{\e}_0(q^2)$ consists of strictly upper triangular skew-Hermitian matrices of format $(\Z\setminus\{0\})\times(\Z\setminus\{0\})$.

For instance, a matrix from $\La^{\e}_2(q^2)$ looks as follows:
$$M = \begin{bmatrix} \ddots &  &  & & & & &  &  & \iddots \\ & 0 & * & * & * & * & * & * & \star &  \\ & 0 & 0 & * & * & * & * & \star & \overline{*} &  \\ & 0 & 0 & m_{-2, -2} & m_{-2, -1} & m_{-2, 1} &  m_{-2, 2} & \overline{*} & \overline{*} &  \\ & 0 & 0 & m_{-1, -2} & m_{-1, -1} & m_{-1, 1} &  m_{-1, 2} & \overline{*} & \overline{*} & \\ & 0 & 0 & m_{1, -2} & m_{1, -1} & m_{1, 1} &  m_{1, 2} & \overline{*} & \overline{*} &  \\ & 0 & 0 & m_{2, -2} & m_{2, -1} & m_{2, 1} &  m_{2, 2} & \overline{*} & \overline{*} & \\ & 0 & 0 & 0 & 0 & 0 & 0 & 0 & \overline{*} & \\  & 0 & 0 & 0 & 0 & 0 & 0 & 0 & 0 &  \\ \iddots & & & & & & & &  & \ddots \end{bmatrix}.$$
The $4\times 4$ submatrix $[m_{i, j}]_{i, j\in\{-2, -1, 1, 2\}}$ belongs to $\uu(4, q^2)$, an asterisk above the secondary diagonal stands for an arbitrary element from $\FF$, an asterisk with a bar below the secondary diagonal means that those elements are determined (conjugate and multiply by $-1$) by those above the diagonal, and a star on the secondary diagonal stands for an element from $\eps\cdot\F$.

Let $\La^{\e}(q^2)$ be the inductive limit group $\varinjlim \La^{\e}_n(q^2)$ arising from the natural inclusions $\La^{\e}_n(q^2) \hookrightarrow\La^{\e}_{n+1}(q^2)$. As a set, $\La^{\e}(q^2)$ consists of the almost strictly upper triangular skew-Hermitian matrices of format $(\Z\setminus\{0\})\times(\Z\setminus\{0\})$.
For $n\in\Z_{\geq 1}$, let $\La^{\e}_{-n}(q^2)\subset\La^{\e}_0(q^2)$ be the subgroup consisting of those matrices $M\in \La^{\e}_0(q^2)$ for which $m_{i, j} = 0$, whenever $|i|\leq n$ and $|j|\leq n$.
We equip the group $\La^{\e}(q^2)$ with the topology in which the subgroups $\La^{\e}_{-n}(q^2)$ form a fundamental system of neighborhoods of $0$.

Each $\La^{\e}_n(q^2)$ is compact and clopen, so $\La^{\e}(q^2)$ is locally compact. Any compact subset of $\La^{\e}(q^2)$ is contained in some $\La^{\e}_n(q^2)$.

The group $U(2\infty, q^2)$ acts on $\La^{\e}(q^2)$ by conjugation; this action preserves the topology of $\La^{\e}(q^2)$.

For any $M\in\Mat^{\e}_{\infty}(q^2)$, denote by $M^{[n]} := [m_{i, j}]_{i, j = -n, \cdots, -1, 1, \cdots, n}$ its central $(2n)\times (2n)$ submatrix. If $M\in\La_n^{\e}(q^2)$, then $M^{[n]}$ belongs to $\uu(2n, q^2)$. A basis for the topology of $\La^{\e}(q^2)$ is given by the elementary cylinder sets
\begin{equation*}
\Cyl_n^{\e}(X) := \{ M \in\La^{\e}(q^2) \mid M\in\La^{\e}_n(q^2),\ M^{[n]} = X \},\quad n\in\Z_{\geq 0},\quad X\in\uu(2n, q^2).
\end{equation*}
If $X_0$ is the unique element of the zero Lie algebra $\uu(0, q^2)$, we agree that $\Cyl^{\e}_0(X_0) = \La^{\e}_0(q^2)$.

Let $\Nil(\La^{\e}_n(q^2))\subset\La^{\e}_n(q^2)$ be the subset of matrices $M = [m_{i, j}]_{i, j\in\Z\setminus\{0\}}$ for which its submatrix $M^{[n]}$ is nilpotent.
Equivalently, $\Nil(\La^{\e}_n(q^2))$ is the union of cylinder sets $\Cyl_n^{\e}(X)$, where $X$ ranges over all nilpotent matrices from $\uu(2n, q^2)$.
Also let
$$\Nil(\La^{\e}(q^2)) := \bigcup_{n=0}^{\infty}{\Nil(\La^{\e}_n(q^2))} \subset\La^{\e}(q^2).$$
Matrices from $\Nil(\La^{\e}(q^2))$ are called pronilpotent matrices. Proposition \ref{prop6.B} shows that
$$
\Nil(\La^{\e}(q^2)) = \bigcup_{u\in U(2\infty, q^2)}{(u\cdot\La_0^{\e}(q^2)\cdot u^{-1})}.
$$
As a result, $\Nil(\La^{\e}(q^2))$ is clopen and $U(2\infty, q^2)$-invariant.

\subsection{The group $U(2\infty+1,q^2)$ and the space $\La^{\o}(q^2)$}

In this subsection, our infinite size matrices are of format $\Z\times\Z$.
Define the spaces $\Mat^{\o}_{\infty}(q^2)$, $\La_n^{\o}(q^2)$ ($n\in\Z$) and $\La^{\o}(q^2) = \varinjlim \La^{\o}_n(q^2)$, in exactly the same way as in the previous subsection, with the only difference that we replace $\e$ by $\o$ everywhere.

For instance, a matrix from $\La^{\o}_1(q^2)$ looks as follows:
$$M = \begin{bmatrix} \ddots &  &  & & & &  &  & \iddots \\ & 0 & * & * & * & * & * & \star &  \\ & 0 & 0 & * & * & * & \star & \overline{*} &  \\ & 0 & 0 & m_{-1, -1} & m_{-1, 0} &  m_{-1, 1} & \overline{*} & \overline{*} &  \\ & 0 & 0 & m_{0, -1} & m_{0, 0} &  m_{0, 1} & \overline{*} & \overline{*} & \\ & 0 & 0 & m_{1, -1} & m_{1, 0} &  m_{1, 1} & \overline{*} & \overline{*} & \\ & 0 & 0 & 0 & 0 & 0 & 0 & \overline{*} & \\  & 0 & 0 & 0 & 0 & 0 & 0 & 0 &  \\ \iddots & & & & & & &  & \ddots \end{bmatrix}.$$
The $3\times 3$ submatrix $[m_{i, j}]_{i, j\in\{-1, 0, 1\}}$ belongs to $\uu(3, q^2)$ (the asterisks, stars and asterisks with a bar have the same meaning as in the example from the previous subsection).

As before, define the topology on $\La^{\o}(q^2)$ by declaring that the subgroups $\La^{\e}_{-n}(q^2)$ ($n\in\Z_{\geq 1}$) form a fundamental system of neighborhoods of $0$. With respect to this topology, each $\La^{\o}_n(q^2)$ is compact and clopen, so $\La^{\o}(q^2)$ is locally compact. Moreover, the conjugation action of $U(2\infty+1, q^2)$ on $\La^{\o}(q^2)$ preserves the topology.

For any $M\in\Mat^{\o}_{\infty}(q^2)$, denote by $M^{[n]} := [m_{i, j}]_{i, j = -n, \dots, 0, \dots, n}$ its central $(2n+1)\times (2n+1)$ submatrix.
If $M\in\La^{\o}_n(q^2)$, then $M^{[n]}\in\uu(2n+1, q^2)$. The elementary cylinder sets, which form a basis for the topology of $\La^{\o}(q^2)$, are \begin{equation*}
\Cyl^{\o}_n(X) := \{ M \in\La^{\o}(q^2) \mid M\in\La^{\o}_n(q^2),\ M^{[n]} = X \},\quad n\in\Z_{\geq 0},\quad X\in\uu(2n+1, q^2).
\end{equation*}

Also define $\Nil(\La_n^{\o}(q^2))$ ($n\in\Z_{\geq 0}$) and $\Nil(\La^{\o}(q^2))$ as in the previous subsection, by replacing $\e$ by $\o$ everywhere. The set $\Nil(\La^{\o}(q^2))$ is clopen and $U(2\infty + 1, q^2)$-invariant. 

\subsection{Invariant Radon measures}
\begin{definition}
Define $\PUeven$ as the convex cone of  $U(2\infty, q^2)$-invariant Radon measures on $\La^{\e}(q^2)$. Likewise, $\PUodd$ is the convex cone of  $U(2\infty+1, q^2)$-invariant Radon measures on $\La^{\o}(q^2)$.
\end{definition}

Let $G$ stand for any of the groups $U(2\infty, q^2)$, $U(2\infty+1, q^2)$. As pointed out in Section \ref{sect1}, one can define a $\GLB$-type topological completion $\overline G\supset G$.
The completion $\overline G$ is defined as the group of infinite size matrices $M$ with finitely many entries below the diagonal and such that $MW_{\infty}M^* = W_{\infty}$, where $W_{\infty}$ is the infinite size matrix with $1$'s in the secondary diagonal and $0$'s elsewhere, and $M^*$ denotes the conjugate transpose of $M$. The topology on $\overline{G}$ can be uniquely characterized as the group topology such that the subgroup of upper triangular matrices contained in $\overline{G}$ (which is a profinite group) is an open subgroup. One can also give an alternative description, as in the case of  $\GLB$, see section \ref{sect3.1}.

The matrices in the topological completion $\overline{G}\supset G$ are of format $(\Z\setminus\{0\})\times(\Z\setminus\{0\})$, if $G = U(2\infty, q^2)$, and of format $\Z\times\Z$, if $G = U(2\infty+1, q^2)$.
We denote these topological completions by $\UB^{\e}\supset U(2\infty, q^2)$ and $\UB^{\o}\supset U(2\infty+1, q^2)$.

The groups $\UB^{\e}$ and $\UB^{\o}$ act by conjugation on $\La^{\e}(q^2)$ and $\La^{\o}(q^2)$, respectively. These are infinite-dimensional versions of the coadjoint action, as discussed in Section \ref{sect1}. The action maps
$$
\UB^{\e}\times\La^{\e}(q^2) \to \La^{\e}(q^2),\qquad \UB^{\o}\times\La^{\o}(q^2) \to \La^{\o}(q^2),
$$
are continuous. Measures from $\PUeven$ and $\PUodd$ are automatically invariant under the action of the larger groups $\UB^{\e}$ and $\UB^{\o}$, respectively: the proof is the same as in Proposition \ref{prop3.A}.  

\begin{definition}
Let $\PUeven_0\subset\PUeven$ be the subset (which is also a convex cone) of measures which are supported on $\Nil(\La^{\e}(q^2))$. Likewise, $\PUodd_0\subset\PUodd$ is the convex cone of  $U(2\infty+1, q^2)$-invariant Radon measures supported on $\Nil(\La^{\o}(q^2))$.
\end{definition}

\begin{remark}[cf. Remark \ref{rem3.A}]
By virtue of Proposition \ref{prop6.B}, nontrivial measures $P$ in $\PUeven_0$ and $\PUodd_0$ can be normalized so that $P(\La^\e_0) = 1$ and $P(\La^\o_0) = 1$, respectively. Their restrictions to $\La^\e_0$ and $\La^\o_0$ can be characterized as \emph{central probability measures} --- the definition can be adapted from \cite[Definition 4.3]{GKV}.
\end{remark}

\subsection{Relationship among $\PUeven$, $\PUodd$, $\PUeven_0$ and $\PUodd_0$}\label{general_measures_U}

This subsection is an analogue of Subsection \ref{sect5.2} ---  we outline how the structures of $\PUeven$ and $\PUodd$ can be understood from $\PUeven_0$ and $\M_0^{U(2\infty+1, q^2)}$. We omit proofs because they copy the arguments from Subsection \ref{sect5.2}.

Recall (see Subsection \ref{sect6.3}) our notation $\wt\T_N$ for the set of conjugacy classes in $\uu(N,q^2)$ and the decomposition $\wt\T_N=\bigsqcup_{s=0}^N{(\wt\Si_s\times \Y_{N-s})}$, where $\wt\Si_s$ denotes the set of nonsingular classes in $\uu(s,q^2)$.
Denote
$$\wt\Si := \bigsqcup_{m=0}^{\infty}{\wt\Si_m},$$
and whenever $\si\in\wt\Si_s$, we set $|\si|:=s$.

For any matrix $M$ from $\La^{\e}(q^2)$ or $\La^\o(q^2)$, an analogue of Lemma \ref{lemma5.A} holds: the nonsingular component $\si_n$ of the submatrix $M^{[n]}$ stabilizes as $n$ gets large; we call it the \emph{stable nonsingular type} and denote it by $\si(M)$; in this way, we obtain two maps $\La^{\e}(q^2)\to\wt\Si$ and $\La^{\o}(q^2)\to\wt\Si$.

Setting
$$
\La^{\e,\si}(q^2):=\{M\in\La^\e(q^2): \si(M)=\si\}, \quad \La^{\o,\si}(q^2):=\{M\in\La^\o(q^2): \si(M)=\si\}, \quad \si\in\wt\Si,
$$
we obtain the partitions
$$
\La^\e(q^2)=\bigsqcup_{\si\in\wt\Si}\La^{\e,\si}(q^2), \quad \La^\o(q^2)=\bigsqcup_{\si\in\wt\Si}\La^{\o,\si}(q^2),
$$
into nonempty invariant clopen subsets.

This in turn entails affine-isomorphisms of convex cones
$$
\PUeven \cong \prod_{\si\in\wt\Si}\M^\e_\si, \qquad\PUodd \cong \prod_{\si\in\wt\Si}\M^\o_\si,
$$
where, by definition, $\M^\e_\si\subset\PUeven$ is formed by the measures supported on $\La^{\e,\si}(q^2)\subset \La^\e(q^2)$, and $\M^\o_\si\subset\PUodd$ is formed by the measures supported on $\La^{\o,\si}(q^2)\subset \La^\o(q^2)$.

Thus, the description of $\PUeven$ and $\PUodd$ is reduced to the description of the cones $\M^\e_\si$ and $\M^\o_\si$, where $\si$ ranges over $\wt\Si$.

Note that if $\si$ is the unique element of $\wt\Si_0$, the corresponding invariant clopen subsets $\La^{\e,\si}(q^2)$ and $\La^{\o,\si}(q^2)$ are precisely the sets of pronilpotent matrices, so the corresponding cones coincide with the cones $\PUeven_0$ and $\PUodd_0$, respectively. The final result is an analogue of Proposition \ref{prop5.B}. 

\begin{proposition}\label{prop7.}
For any $\si\in\wt\Si$, the following affine-isomorphisms of convex cones hold
$$
\M^\e_\si\cong \begin{cases}
\PUeven_0, & \text{\rm for $|\si|$ even},\\
\PUodd_0, & \text{\rm for $|\si|$ odd};
\end{cases} \qquad 
\M^\o_\si\cong \begin{cases}
\PUodd_0, & \text{\rm for $|\si|$ even},\\
\PUeven_0, & \text{\rm for $|\si|$ odd}.
\end{cases}
$$
\end{proposition}

The conclusion is that the structures of $\PUeven$ and $\PUodd$ are determined by $\PUeven_0$ and $\PUodd_0$. Finally, the study of the latter convex cones is undertaken in the next section.

\section{The branching graphs of $\PUeven_{0}$, $\PUodd_0$ and Ennola's duality}\label{sec:ennola}

The main results of this section are two theorems.

Theorem \ref{thm8.A} contains a computation with HL functions; its significance is that it implies the existence of new branching graphs, which are based on HL functions with \emph{negative} parameter $t$. We call them the even and odd HL-deformed Young graphs.

Theorem \ref{Ustructure}  translates the problem of characterizing $\PUeven_0$ and $\PUodd_0$ into the problem of describing the convex cone of nonnegative harmonic functions on these graphs.

As applications, we exhibit examples of measures belonging to $\PUeven_0$ and $\PUodd_0$.

\subsection{The branching graphs $\Ga^{U(2\infty, q^2)}$ and $\Ga^{U(2\infty+1, q^2)}$}\label{branching_U}

Let $n\in\Z_{\geq 0}$ be arbitrary, and set $N = 2n$ or $N = 2n+1$.
Let $\mu\in\Y_{N}$ and $X\in\Nil(\uu(N, q^2))$ be nilpotent of Jordan type $\mu$.
Let $\la\in\Y_{N+2}$ and denote by $\wt{\Li}^{n+1}_n(\la, \mu)$ the number of (nilpotent, skew-Hermitian) matrices of the form
\begin{equation}\label{Yform}
Y := \begin{bmatrix} 0 & -x^*W_{N} & \eps y \\ 0 & X & x \\  0 & 0 & 0 \end{bmatrix},\quad x\in\FF^{N},\ y\in\F,
\end{equation}
which are of type $\la$. Recall that $W_{N}\in\Mat_N(\FF)$ was defined in \eqref{matrixW}. The quantity $\wt{\Li}^{n+1}_n(\la, \mu)$ does not depend on the specific choice of matrix $X$ of Jordan type $\mu$.

Our notation $\wt{\Li}^{n+1}_n(\la, \mu)$ is somewhat imprecise because it is unclear whether $|\mu| = 2n$, $|\la| = 2n+2$, or $|\mu| = 2n+1$, $|\la| = 2n+3$. However, this will not cause any issue.

\begin{definition}\label{def8.A}
Let $\mu\in\Y_N$ and $\la\in\Y_{N+2}$, where $N\ge0$. We write $\mu\nearrow\!\nearrow\la$ if $\mu\subset\la$ and the two boxes of the skew diagram $\th := \la\setminus\mu$ either lie in a single column (so that $\th$ is a vertical domino) or lie in two consecutive columns (in particular, $\th$ may be a horizontal domino).
\end{definition}

The following result gives explicit formulas for $\wt{\Li}^{n+1}_n(\la, \mu)$. This is an analogue of Proposition \ref{prop_borodin}, but the proof is more laborious; it is deferred to Section \ref{sect9}.  

\begin{proposition}\label{Lmula}
Let $n\in\Z_{\geq 0}$, and set $N = 2n$ or $N = 2n+1$. Let $\mu\in\Y_{N}$, $\la\in\Y_{N+2}$ be arbitrary. Recall the notation $m_i(\mu)$ from \eqref{eq2.A}.

{\rm(i)} The quantity $\wt{\Li}^{n+1}_n(\la, \mu)$ is nonzero if and only if $\mu\nearrow\!\nearrow\la$.

{\rm(ii)} Suppose $\mu\nearrow\!\nearrow\la$. The value of $\wt{\Li}^{n+1}_n(\la, \mu)$ depends on the columns where the boxes of $\th=\la\setminus\mu$ lie:

\smallskip

{\rm(1)} if $\th$ lies in a single column, with number $k$, then
$$
\wt{\Li}_n^{n+1}(\la, \mu) = \begin{cases}
q^{2N - 2\sum_{j \geq k}{m_j}(\mu)} \cdot(1 - (-q)^{-m_{k-1}(\mu)})(1 - (-q)^{1-m_{k-1}(\mu)}), 
& k>1,\\
q^{2N - 2\sum_{j \geq 1}{m_j(\mu)}}, & k=1;
\end{cases}
$$
\smallskip

{\rm(2)} otherwise, if $\th$ lies in two consecutive columns, with numbers $k$ and $k+1$,  then
$$
\wt{\Li}_n^{n+1}(\la, \mu) = \begin{cases}
q^{2N - 2\sum_{j \geq k}{m_j(\mu)}} \cdot (q - 1)(1 - (-q)^{-m_{k-1}(\mu)}), & k>1,\\
q^{2N - 2\sum_{j \geq 1}{m_j(\mu)}} \cdot (q - 1), & k = 1.
\end{cases}
$$
\end{proposition}

In both variants, (1) and (2), the second formula (the one with $k=1$) can be viewed as a particular case of the first formula, provided we agree that $m_0(\mu)=+\infty$ and $(-q)^{-m_0(\mu)}:=0$. 

Note also that in variant (1), the right-hand side vanishes if $k>1$ and $m_{k-1}(\mu)\le1$, which agrees with the fact that in such a case appending a vertical domino to the $k$th column of $\mu$ is impossible. Likewise, in variant (2), the right-hand side vanishes if $k>1$ and $m_{k-1}(\mu)=0$, which agrees with the fact that then appending a box to the $k$th column is impossible.

\begin{definition}
The branching graph $\Ga^{U(2\infty, q^2)}$ (resp. $\Ga^{U(2\infty+1, q^2)}$) is defined by the following data:

--- the set $\Y_{\e}$ (resp. $\Y_{\o}$) of partitions of even size (resp. odd size) is the set of vertices, and the disjoint union $\Y_{\e} = \bigsqcup_{n \geq 0}{\Y_{2n}}$ (resp. $\Y_{\o} = \bigsqcup_{n \geq 0}{\Y_{2n+1}}$) defines a grading on the vertices;

--- an edge connects $\mu\in\Y_{2n}$ and $\la\in\Y_{2n+2}$ (resp. $\mu\in\Y_{2n+1}$ and $\la\in\Y_{2n+3}$) iff $\mu\nearrow\!\nearrow\la$; the associated edge weight is $\wt{\Li}^{n+1}_n(\la, \mu)$.
\end{definition}

This formulation is justified by the following lemma. Recall that we adopt the definition of branching graphs given in Definition \ref{def4.A1}. 

\begin{lemma}\label{lemma8.A}
The graded graphs with the vertex sets $\Y_\e$ and $\Y_\o$ and the edges $\mu\nearrow\!\nearrow\la$ satisfy the conditions listed in Definition \ref{def4.A1}.  
\end{lemma}

\begin{proof}
The only condition which is not evident is the following one: for any vertex $\la$ of level $n\geq 1$ (meaning that $\la\in \Y_{2n}$ or $\la\in\Y_{2n+1}$, depending on the parity), there should exist a vertex $\mu\nearrow\!\nearrow\la$.

In fact, we can check that there exists a path $\mu^{(0)}\nearrow\!\nearrow\cdots\nearrow\!\nearrow\mu^{(n)}=\la$ in the graph, where the starting point $\mu^{(0)}$ is the empty diagram or the one-box diagram, depending on the parity.
To check this, use the known fact that an arbitrary Young diagram can be reduced to its \emph{$2$-core} (which is either empty or a staircase shape $(m,m-1,\dots,1)$) by consecutive removal of $2$-rim hooks (the latter are vertical or horizontal dominoes), see e.g. \cite[solution to Exercise 7.59, item g]{StF} and references therein. Next, from a staircase shape of size $3$ or more, one can remove two boxes lying in two consecutive columns; after that the procedure can be iterated; in the end we achieve the root, which is either the empty diagram or the one-box diagram.
\end{proof}

The following lemma is proved exactly in the same way as Lemma \ref{lemma4.B}.

\begin{lemma}\label{lemma8.B}
There is a bijective correspondence between measures $P\in\PUeven_0$ (resp.,  $P\in\PUodd_0$) and nonnegative harmonic functions $\vp^\Ga$ on the branching graph $\Ga=\Ga^{U(2\infty,q^2)}$ (resp., $\Ga=\Ga^{U(2\infty+1,q^2)}$), uniquely characterized by the property that for any $N=2n$ (resp. $N=2n+1$) and any nilpotent matrix $X\in\uu(N,q^2)$ of a given Jordan type $\mu\in\Y_N$, one has
\begin{equation}\label{eq8.E}
P(\Cyl_N(X))=\vp^\Ga(\mu).
\end{equation}
In the equation \eqref{eq8.E}, $\Cyl_N(X)$ stands for $\Cyl_N^{\e}(X)$ if $P\in\PUeven_0$, and for $\Cyl_N^{\o}(X)$ if $P\in\M_0^{U(2\infty+1, q^2)}$.
\end{lemma}

\subsection{Even and odd HL-deformed Young graphs}

Recall the $Q$-HL functions $Q_{\la}(; t)$, $\la\in\Y$, discussed in Section \ref{HL_sec}. In that section, the parameter $t$ belonged to the interval $(0, 1)$; here, we need that $t \in (-1, 0)$.

It will be convenient to slightly modify the $Q$-HL functions as follows:
$$
\wt{Q}_{\la}(; t) := (-1)^{n(\la)}Q_{\la}(; t),\quad\la\in\Y.
$$
Given $\mu\in\Y$, consider the expansion 
 \begin{equation*}
\left((1 - t^2)p_2 \right) \cdot \wt{Q}_\mu(;t)=\sum_{\la}{\xi_{\la/\mu}(t) \wt{Q}_\la(;t)}.
\end{equation*}
Here, $\la$ ranges over the set of Young diagrams with $|\la|=|\mu|+2$ and $\xi_{\la/\mu}(t)$ are certain coefficients.

In the next theorem we compute these coefficients. Let, as usual, $m_k(\mu)$ denotes the number of rows in $\mu$ of a given length $k$. We also agree that $m_0(\mu) = +\infty$ and so $t^{m_{k-1}(\mu)}=t^{m_{k-1}(\mu)-1}=0$ for $k=1$.

\begin{theorem}\label{thm8.A}
The coefficient $\xi_{\la/\mu}(t)$ vanishes unless $\mu\nearrow\!\nearrow\la$. Next, suppose $\mu\nearrow\!\nearrow\la$ and set $\th=\la\setminus\mu$; then we have: 

\smallskip

{\rm(1)} if $\th$ lies in a single column, with number $k$, then
$$
\xi_{\la/\mu}(t)=(1-t^{m_{k-1}(\mu)})(1-t^{m_{k-1}(\mu)-1});
$$

\smallskip

{\rm(2)} if $\th$ lies in two consecutive columns, with numbers $k$ and $k+1$, then
$$
\xi_{\la/\mu}(t)=(-t)^{m_k(\mu)}(1+t)(1-t^{m_{k-1}(\mu)}).
$$
\end{theorem}

Here is an immediate corollary.

\begin{corollary}\label{cor8.A}
Suppose $t\in(-1, 0)$. Then the coefficients $\xi_{\la/\mu}(t)$ are strictly positive for each pair of diagrams $\mu\nearrow\!\nearrow\la$.
\end{corollary}

\begin{proof}[Proof of Theorem \ref{thm8.A}]
From the definition of the one-row HL functions $Q_{(r)}(;t)$, it follows that
\begin{equation*}
(1-t^2)p_2=2Q_{(2)}(;t)-(Q_{(1)}(;t))^2.
\end{equation*}
Next, in the expansion 
\begin{equation*}
Q_{(r)}(;t)Q_\mu(;t)=\sum_{\la}\psi_{\la/\mu}(t) Q_\la(;t), \qquad r=1,2,\dots,
\end{equation*}
the coefficients $\psi_{\la/\mu}(t)$ vanish unless $\la\succ\mu$, meaning that $\la\supset\mu$ and $\la\setminus\mu$ is a horizontal strip (of length $r$), see \cite[Ch. III, (5.7$'$)]{Mac}. It follows that
\begin{equation}\label{eq8.A}
\xi_{\la/\mu}(t)=(-1)^{n(\la)-n(\mu)}\left(2\psi_{\la/\mu}(t)-\sum_{\nu: \,\mu\nearrow\nu\nearrow\la}\psi_{\nu/\mu}(t)\psi_{\la/\nu}(t)\right).
\end{equation}

Below we use the recipe for computing the coefficients $\psi_{\la/\mu}(t)$ for $\la\succ\mu$, explained in \cite[Ch. III, (5.8$'$)]{Mac}. We examine three possible cases.

\medskip

1. \emph{$\th$ is a vertical domino lying in column $k$}.  Then $\th$ is not a horizontal strip, so that $\psi_{\la/\mu}=0$. Next, there is a single $\nu$ situated between $\mu$ and $\la$, and 
$$
\psi_{\nu/\mu}(t)=1-t^{m_{k-1}(\mu)}, \quad \psi_{\la/\nu}(t)=1-t^{m_{k-1}(\mu)-1}.
$$
Finally, $n(\la)-n(\mu)=2\mu'_k+1$ is odd. This leads  to the formula in (1).

\medskip

2. \emph{$\th$ lies in two consecutive columns, with numbers $k$ and $k+1$}. Then $\psi_{\la/\mu}=1-t^{m_{k-1}(\mu)}$. Next, there are two intermediate diagrams $\nu$ and we have
$$
\sum_{\nu: \,\mu\nearrow\nu\nearrow\la}\psi_{\nu/\mu}(t)\psi_{\la/\nu}(t)=
(1-t^{m_{k-1}(\mu)})(1-t^{m_k(\mu)+1})+(1-t^{m_k(\mu)})(1-t^{m_{k-1}(\mu)})
$$
Finally, we have $n(\la)-n(\mu)=\mu'_k+\mu'_{k+1}$, which has the same parity as $\mu'_k-\mu'_{k+1}=m_k(\mu)$. This leads to the formula in (2).

\medskip

3. \emph{$\th$ lies in two columns, with numbers $k$ and $\ell$, where $\ell>k+1$}. We claim that in this case the difference in \eqref{eq8.A} is equal to $0$, so that $\xi_{\la/\mu}(t)=0$. Indeed, we have 
$$
2\psi_{\la/\mu}(t)=2(1-t^{m_{k-1}(\mu)})(1-t^{m_{\ell-1}(\mu)}).
$$
Next, there is two variants for $\nu$, and each of them produces the above expression, without the prefactor $2$. Thus, the difference in \eqref{eq8.A} vanishes.

This completes the proof.
\end{proof}

By virtue of Corollary \ref{cor8.A} and  Lemma \ref{lemma8.A}, the following definition makes sense (cf. Definition \ref{def4.HL}). 

\begin{definition}
Let $t\in(-1,0)$ be arbitrary. The \emph{even HL-deformed Young graph} $\Y^{\HL}_{\e}(t)$ (resp. \emph{odd HL-deformed Young graph} $\Y^{\HL}_{\o}(t)$) is the branching graph given by:

--- $\Y_{\e} = \bigsqcup_{n \geq 0}{\Y_{2n}}$ (resp. $\Y_{\o} = \bigsqcup_{n \geq 0}{\Y_{2n+1}}$) is the graded set of vertices;

--- an edge connects $\mu\in\Y_{2n}$ and $\la\in\Y_{2n+2}$ (resp. $\mu\in\Y_{2n+1}$ and $\la\in\Y_{2n+3}$) iff $\mu\nearrow\!\nearrow\la$, and then the corresponding edge weight is $\xi_{\la/\mu}(t)$.
\end{definition}

\subsection{Final result}

\begin{theorem}\label{Ustructure}

{\rm(i)}  The measures $P\in\PUeven_{0}$ {\rm(}resp., $P\in\PUodd_0${\rm)} are in one-to-one correspondence with the nonnegative harmonic functions $\vp$ on the graphs $\Y^{\HL}_{\e}(-q^{-1})$ {\rm(}resp., $\Y^{\HL}_{\o}(-q^{-1})${\rm)}.  

This correspondence $P\leftrightarrow \varphi$ is uniquely determined by the property that the mass of an elementary cylinder set $\Cyl_N(X)$, where $N=2n$ is even {\rm(}resp., $N=2n+1$ is odd{\rm)} and $X\in\Nil(\uu(N,q^2))$ is a nilpotent matrix of Jordan type $\la\in\Y_N$, is given by 
$$
P(\Cyl_N(X))=q^{n(\la)-N(N-1)/2}\varphi(\la).
$$ 
Here, $\Cyl_N(X)$ stands for $\Cyl_N^{\e}(X)$ if $P\in\PUeven_0$ and for $\Cyl_N^{\o}(X)$ if $P\in\PUodd_0$.

{\rm(ii)} The correspondence $P\leftrightarrow \varphi$ establishes affine-isomorphisms of convex cones,
$$
\PUeven_0 \leftrightarrow \Harm_+(\Y^{\HL}_{\e}(-q^{-1})), \qquad
\PUodd_0 \leftrightarrow \Harm_+(\Y^{\HL}_{\o}(-q^{-1})).
$$
In particular, ergodic measures $P$ correspond precisely to extreme harmonic functions $\vp$. 
\end{theorem}

We need a lemma linking the quantities $\wt{\Li}^{n+1}_n(\la, \mu)$ computed in Proposition \ref{Lmula} with the coefficients $\xi_{\la/\mu}(t)$ computed in Theorem \ref{thm8.A}. Recall the function defined in \eqref{eq4.f}:
$$
f(\la)=q^{n(\la)-\binom{|\la|}2}, \quad \la\in\Y,
$$
 
\begin{lemma}\label{lemma8.C}
Let $\mu\in\Y_N$, $\la\in\Y_{N+2}$, where $N=2n$ or $N=2n+1$. If $\mu\nearrow\!\nearrow\la$, then 
\begin{equation}\label{xiN}
\wt{\Li}^{n+1}_n(\la, \mu)= \xi_{\la/\mu}(-q^{-1})\cdot\frac{f(\mu)}{f(\la)} =
\xi_{\la/\mu}(-q^{-1})\cdot q^{n(\mu) - n(\la) + 2N+1}.
\end{equation}
\end{lemma}

\begin{proof}
This follows from a comparison between the formulas of Proposition \ref{Lmula} and those of Theorem \ref{thm8.A}. There are two variants that we denoted as (1) and (2). We use the relations
$$
n(\la)-n(\mu)=\begin{cases} 
2\mu'_k+1 & \text{in the case $(1)$}\\
\mu'_k+\mu'_{k+1}=2\mu'_k-m_k(\mu) & \text{in the case $(2)$}
\end{cases}
$$
and
$$
\sum_{j\ge k}m_j(\mu)=\mu'_k.
$$
\end{proof}

\begin{proof}[Proof of Theorem \ref{Ustructure}]
(i) This follows from the chain of bijections
$$
P\leftrightarrow \vp^\Ga \leftrightarrow \vp,
$$
where the first bijection is given by Lemma \ref{lemma8.B} and the second bijection is given by Lemma \ref{lemma8.C}, which establishes the similarity of the branching graphs with edge weights $\wt L^{n+1}_n(\la,\mu)$ and $\xi_{\la/\mu}(-q^{-1})$. 

(ii) This claim is a consequence of (i). 
\end{proof}

It is interesting to compare  Proposition \ref{HLgraph} (for the branching graph $\Y^{\HL}(t)$ related to $GL(\infty, q)$) and Theorem \ref{Ustructure} (for the branching graphs $\Y^{\HL}_{\e}(t)$, $\Y^{\HL}_{\o}(t)$ related to $U(2\infty, q^2)$, $U(2\infty+1, q^2)$).
The formulas specifying the link between invariant measures and harmonic functions on HL-deformed graphs look identical, with the main difference being that in the former case, the HL parameter $t$ is specialized to $q^{-1}$, while in the latter case, it is specialized to $-q^{-1}$.

It is known that the sign flip $q \leftrightarrow -q$ arises in the representation theory of the finite unitary groups. Namely,  the images of the irreducible characters of $GL(n, q)$ and $U(n, q^2)$, under appropriate characteristic maps, coincide after the sign flip $q \leftrightarrow -q$; see \cite{E1}, \cite{E2}, \cite{TV}. This phenomenon is called \emph{Ennola's duality}. 
Our results suggest that a version of Ennola's duality might exist in our infinite-dimensional setting.

It is an open problem to describe explicitly the set of nonnegative harmonic functions on the graphs $\Y^\HL_\e(t)$ and $\Y^\HL_\o(t)$ with negative $t\in(-1,0)$, as it was done in Proposition \ref{prop4.B} for the graph $\Y^\HL(t)$ with positive $t\in(0,1)$ (Problem \ref{problem1.B} from the introduction). 

\begin{remark}
For the HL-deformed graph $\Y^{\HL}(t)$ (in fact, for more general branching graphs), Kerov was able to obtain a list of extreme nonnegative harmonic functions, and conjectured the completeness of this list. A special case of his construction is described in detail in \cite[Sect. 4]{GO}. The conjecture was proved much later by Matveev \cite{Mat}; in the case of our interest, the result is stated in Proposition \ref{prop4.B}.
Kerov's construction is based on the coalgebra structure of the ring of symmetric functions.
It is unclear whether his method can be adapted to the graphs $\Y^\HL_\e(t)$ and $\Y^\HL_\o(t)$.
\end{remark}

\subsection{Examples}

The problem just stated can be reformulated as follows (cf. Remark \ref{rem4.A}). Consider the decomposition 
$\Sym=\Sym^\e\oplus \Sym^\o$, where $\Sym^\e$ and $\Sym^\o$ are the linear subspaces spanned by the homogeneous elements of even and odd degree, respectively. Next, for $t\in(-1,0)$, let $C^\HL_\e(t)\subset \Sym^\e$ and $C^\HL_\o(t)\subset\Sym^\o$ be the convex cones spanned by the functions $\wt Q_\la(;t)$, where $|\la|$ is assumed to be even or odd, respectively.
Then we are interested in linear functionals  $\Phi$ on $\Sym^\e$ or $\Sym^\o$, subject to the following conditions:
\begin{gather}
\text{$\Phi(p_2 F)=\Phi(F)$ for any $F\in\Sym^\e$ or $F\in\Sym^\o$ ($p_2$-harmonicity);} \label{eq8.C} \\
\text{$\Phi$ is nonnegative on $C^\HL_\e(t)$ or $C^\HL_\o(t)$, respectively (positivity)}. \label{eq8.D}
\end{gather}

These  functionals form convex cones which are in a natural bijective correspondence with the cones of nonnegative harmonic functions. Adding the extra normalization condition $\Phi(1)=1$ if $\Phi:\Sym^\e\to\R$, or $\Phi(p_1)=1$ if $\Phi:\Sym^\o\to\R$, we obtain bases of the cones. 

In this subsection, we construct examples of such functionals for a general value of $t\in(-1, 0)$.  For $t=-q^{-1}$, the functionals give rise to invariant measures, via the relation
\begin{equation}\label{measures_spec}
P(\Cyl_N(X)) = \frac{q^{n(\la) - \frac{N(N-1)}{2}}}{(1 - q^{-2})^{\frac{N}{2}}} \cdot \Phi( \wt{Q}_{\la}(; -q^{-1}) ),
\end{equation}
where $X\in\uu(N, q^2)$ is any nilpotent matrix of Jordan type $\la$.

\subsubsection{Plancherel-type functionals}\label{sec:plancherel}

Consider the basis $\{p_\rho: \rho\in\Y\}$ in $\Sym$ formed by the products of power-sums. We define the functionals $\Phi^\Planch_\e:\Sym^\e\to\R$ and $\Phi^\Planch_\o:\Sym^\e\to\R$ as follows (below $n=0,1,2,\dots$):
\begin{equation*}
\Phi^\Planch_\e(p_\rho)=\begin{cases} 1, & \rho=(2^n), \\
0, & \text{otherwise}; 
\end{cases}
\qquad
\Phi^\Planch_\o(p_\rho)=\begin{cases} 1, & \rho=(2^n, 1), \\
0, & \text{otherwise}.
\end{cases}
\end{equation*}

\begin{proposition}
These functionals satisfy the conditions \eqref{eq8.C}-- \eqref{eq8.D}.
\end{proposition}

\begin{proof}

The $p_2$-harmonicity property is clear from the very definition. Let us check the positivity. In fact we will prove a stronger claim:
\begin{equation}\label{eq8.B}
\Phi^\Planch_\e(\wt Q_\la(;t))>0, \quad \la\in\Y_{2n}; \qquad \Phi^\Planch_\o(\wt Q_\la(;t))>0, \quad \la\in\Y_{2n+1}.
\end{equation}

Below $(\ccdot,\ccdot)$ is the scalar product in $\Sym$ depending on the HL parameter $t$, see \cite[Ch. III, Sect. 4]{Mac}. We use the fact that both $\{Q_\la(;t): \la\in\Y\}$ and $\{p_\rho: \rho\in\Y\}$ are orthogonal bases. 

Examine the even case. For $\la\in\Y_{2n}$, we have 
$$
\Phi^\Planch_\e(\wt Q_\la(;t))=\frac{(\wt Q_\la(;t), p_2^n)}{(p_2^n, p_2^n)}.
$$
By \cite[Ch. III, (4.11)]{Mac},
$$
(p_2^n,p_2^n)=(1-t^2)^{-n}\cdot 2^n n!>0.
$$
Next, let $\xi_{\la/\varnothing}$ stand for the coefficient of $\wt Q_\la(;t)$ in the expansion 
$$
\left((1-t^2)p_2\right)^n=\sum_{\nu\in\Y_{2n}}\xi_{\nu/\varnothing}(t) \wt Q_\nu(;t). 
$$
Then we have
$$
(\wt Q_\la(;t), p_2^n)=\frac{\xi_{\la/\varnothing}(t)b_{\la}(t)}{(1 - t^2)^n}, \qquad b_\la(t):=(Q_\la(;t), Q_\la(;t))=(\wt Q_\la(;t), \wt Q_\la(;t)).
$$

A formula for the quantity $b_{\la}(t)$ was previously displayed in \eqref{b_def}; it clearly shows that $b_\la(t)>0$, if $t\in (-1, 0)$. It remains to prove that $\xi_{\la/\varnothing}(t)>0$. But this coefficient equals the sum of the weights of all paths $\varnothing\nearrow\!\nearrow\dots\nearrow\!\nearrow\la$ joining $\la$ to the root $\varnothing$ in the graph $\Y^\HL_\e$, where the weight of a path is equal to the product of its edge weights. By Lemma \ref{lemma8.A}, the set of these paths is nonempty, which entails the desired inequality.

In the odd case the argument is similar. In this case the root of $\Y^\HL_\o$ is the one-box diagram.
\end{proof}

\subsubsection{Functionals connected with the principal specialization} 
Consider the specializations $\Sym\to\R$ defined by
\begin{equation}\label{eq8.E2}
\Phi_m(F):=F(c_m,c_m t,\dots,c_mt^{m-1}), \ m=1,2,\dots; \quad \Phi_\infty(F):=F(c_\infty,c_\infty t,c_\infty t^2,\dots),
\end{equation}
where the positive constants $c_1,c_2,\dots, c_\infty$ are:
$$
c_m = \sqrt{\dfrac{1-t^2}{1-t^{2m}}}, \qquad c_\infty = \sqrt{1-t^2}.
$$
We are interested in the restrictions of $\Phi_m$ and $\Phi_{\infty}$ to $\Sym^\e$ and to $\Sym^\o$.

\begin{proposition}
The functionals $\Phi_m$, $\Phi_\infty$ satisfy the conditions \eqref{eq8.C}-- \eqref{eq8.D},  both for the even and odd cases.
\end{proposition}

\begin{proof}
The $p_2$-harmonicity property follows from our choices of constants $c_m$ and $c_\infty$, which imply that $\Phi_m(p_2) = \Phi_\infty(p_2) = 1$. The positivity property follows from the fact that (see \cite[Ch. III, Sect. 4, ex. 3]{Mac})
$$
\Phi_m(Q_\la(;t)) = \begin{cases} c_m^{|\la|} t^{n(\la)} \prod\limits_{i=1}^{\ell(\la)}(1-t^{m-i+1}), & \ell(\la)\le m,\\
0, & \ell(\la)>m;
\end{cases}
\qquad
\Phi_\infty(Q_\la(;t)) = c^{|\la|}_{\infty} t^{n(\la)}.
$$
Since $(-1)^{n(\la)}t^{n(\la)}>0$ for negative $t$, we obtain that $\Phi_m(\wt Q_\la(;t))$ is strictly positive for $\ell(\la)\le m$ and is vanishing otherwise, whereas $\Phi_\infty(\wt Q_\la(;t))>0$ for all $\la$. 
\end{proof}

\begin{remark}
The invariant measure $P$ corresponding to $\Phi_\infty$ is the restriction of the (properly normalized) Haar measure to the subset of pronilpotent matrices, both in even and odd cases (cf. Example \ref{ex4.Haar}).
In fact, the corresponding measure $P$ is determined by the relations (see \eqref{measures_spec}):
$$
P(\Cyl_N(X)) = q^{\frac{-N(N-1)}{2}},\quad X\in\uu(N, q^2),\quad N = 0, 1, \cdots.
$$
This relation implies that $P$ is the restriction of a Haar measure because $P(\Cyl_N(X))$ does not depend on the Jordan type of $X$.
\end{remark}

\subsubsection{One more family of functionals}
The functional $\Phi_2$ from \eqref{eq8.E2} is the specialization  $F\mapsto F(\frac1{\sqrt{1+t^2}}, \frac{t}{\sqrt{1+t^2}})$. The following functionals form a one-parameter deformation of this specialization.  

Let $a_1, a_2$ be real numbers such that $a_1>0>a_2$, $a_1^2+a_2^2=1$, and $a_1>|a_2|$. Let $\Phi_{a_1,a_2}: \Sym\to\R$ be the specialization $F\mapsto F(a_1,a_2)$. 

\begin{proposition}
The functionals $\Phi_{a_1,a_2}$ satisfy the conditions \eqref{eq8.C}-- \eqref{eq8.D},  both for the even and odd cases.
\end{proposition}

\begin{proof}
The $p_2$-harmonicity property holds, because $\Phi_{a_1,a_2}(p_2)=a_1^2+a_2^2=1$. Next, we have $\Phi_{a_1,a_2}(\wt Q_\la(;t))=0$ if $\ell(\la)>2$. We are going to show that  $\Phi_{a_1, a_2}(\wt{Q}_{\la}(;t))\ge0$ for all $\la$ with $\ell(\la)\le2$.

This is clear for $\la=\varnothing$. Suppose $\ell(\la)=1$, so that $\la=(m)$ with $m=1,2,\dots$. Then $n(\la)=0$, so that we have to check that $Q_{(m)}(a_1,a_2;t)>0$. From \cite[Ch. III, (2.9)]{Mac},
\begin{equation}\label{Q_two}
Q_{(m)}(a_1, a_2; t) = (1-t)\cdot\left( a_1^{m}\cdot\frac{a_1 -ta_2}{a_1 - a_2} + a_2^{m}\cdot\frac{a_2-ta_1}{a_2 - a_1} \right)
\end{equation}
Since $1-t>0$ and $a_1-a_2>0$, the desired inequality reduces to  
\begin{equation}\label{easy_ineq}
a_1^{m}(a_1 -ta_2) - a_2^{m}(a_2 -t a_1) \stackrel{?}{>} 0.
\end{equation}

Since the left hand side of \eqref{easy_ineq} is homogeneous on $a_1, a_2$, we can assume that $a_2 = -1$ and $a_1=a>1$. 
Then the inequality turns into
$$
a^{m+1}+(-1)^m\stackrel{?}{>} -t(a^m+(-1)^m a).
$$
Since $t\in(-1,0)$ and $a>1$, it suffices to prove that
$$
a^{m+1}+(-1)^m \stackrel{?}{\ge} a^m+(-1)^m a,
$$
but this follows from the evident inequality
$$
a^m(a-1)\ge (-1)^m (a-1).
$$

It remains to examine the case $\ell(\la) = 2$. Then $\la = (m,n)$ for some integers $m\ge n\ge1$. Observe that  $n(\la)=n$; using this we obtain
\begin{multline*}
\wt Q_{(m,n)}(a_1,a_2;t)=(-1)^n Q_{(m,n)}(a_1,a_2;t)=(-1)^n(a_1a_2)^n Q_{(m-n)}(a_1,a_2;t)\\
=|a_1a_2|^nQ_{(m-n)}(a_1,a_2;t)>0,
\end{multline*}
where the last inequality holds by the case already considered.
\end{proof}

\section{Proof of Propositions \ref{prop_borodin} and \ref{Lmula}}\label{sect9}

\subsection{Two lemmas}\label{sect9.1}

Let $V$ be a finite-dimensional vector space. Given an operator $A$ on $V$, we denote by $\Ran A$ its range and set $\rk A:=\dim(\Ran A)$. 

The following lemma shows how to find the Jordan type of a nilpotent operator.

\begin{lemma}\label{lemma9.A}
Let $X$ be a nilpotent operator on $V$,  $\mu\in\Y$ be its Jordan type, and $\mu'$ be the transposed diagram. Then
$$
\rk X^{i-1} - \rk X^i = \mu'_i,\text{ for all }i \geq 1.
$$
As a consequence,
$$
\rk X^k = \sum_{j \geq k+1}{\mu'_j},\text{ for all }k \geq 0.
$$
\end{lemma}

\begin{proof}
Easy exercise using the fact that if $A$ is a $p\times p$ nilpotent matrix of Jordan type $(p)$, then $\rk A^k=\max(p-k,0)$ (cf. Macdonald \cite[Ch. II, (1.4)]{Mac}). 
\end{proof}

Let $X$ be as in Lemma \ref{lemma9.A}. We assign to it the nested sequence of subspaces 
$$
\{0\}=V_0\subseteq V_1\subseteq V_2\subseteq V_3\subseteq\dots,
$$
where 
\begin{equation}\label{eq9.D}
V_k := \{ v\in V : X^{k-1}v \in\Ran X^k \},  \quad k=1,2,\dots
\end{equation}
(the fact that  $V_k\subseteq V_{k+1}$ follows directly from the very definition). 

\begin{lemma}\label{lemma9.B}
Let $X$ be as above and $\mu$ be its Jordan type. We have
$$
\dim V_k = \dim V - \sum_{j \geq k}{m_j}, \quad k\ge0,
$$
where, as before, $m_j=m_j(\mu):=\#\{i: \mu_i=j\}$.
\end{lemma}

In particular, we see that $V_k=V$ for $k\ge\mu_1+1$.

\begin{proof}
Examine first the special case when $\mu=(p)$, where $p:=\dim V$. Pick a basis $v_1,\dots,v_p$ of $V$ such that $X v_i=v_{i-1}$ for $i=1,\dots,p$, where $v_0:=0$. Then 
$$
V_k=\begin{cases} \operatorname{span}\{v_1,\dots,v_{p-1}\}, & 1\le k\le p, \\
V, & k\ge p+1.
\end{cases}
$$

More generally, if $\mu$ is arbitrary, we pick a basis $\{v_{ij}\}$ indexed by the boxes $(i,j)\in\mu$ such that $X v_{ij}=v_{i,j-1}$ with the understanding that $v_{i0}:=0$ (see the proof of Proposition \ref{prop6.A}). Then $V_k$ is the linear span of the basis vectors $v_{ij}$ such that $(i,j)$ satisfies one of the following two conditions:

\smallskip
$\bullet$ $\mu_i\le k-1$, or

\smallskip
$\bullet$ $\mu_i\ge k$ and $1\le j\le \mu_i-1$.

\smallskip

The number of such pairs $(i,j)$ equals $|\mu|-\sum_{j \geq k}{m_j}$, which completes the proof. 
\end{proof}

\subsection{Proof of Proposition \ref{prop_borodin}}

We set $V=\F^n$ and fix a nilpotent operator $X$ on $V$, which we identify with the corresponding $n\times n$ matrix. Let  $\mu\in\Y_n$ denote the Jordan type of $X$. Next, given $x\in V$, we consider the matrix 
\begin{equation*}
Y := \begin{bmatrix} X & x\\ 0 & 0 \end{bmatrix}
\end{equation*}
and denote by $\la$ its Jordan type. Our task is  to compute the number of vectors $x\in V$ leading to a given $\la\in\Y_{n+1}$; let us denote this quantity by $L(\la,\mu)$ (instead of the more detailed notation $L^{n+1}_n(\la,\mu)$, as in the original formulation of the proposition). 

Set
\begin{equation*}
\eps_i:=\rk Y^i -\rk X^i, \quad i=0,1,2,\dots\,.
\end{equation*}
By Lemma \ref{lemma9.A},
\begin{equation}\label{eq9.A}
\la'_i-\mu'_i=\eps_{i-1}-\eps_i, \quad i=1,2,3,\dots\,.
\end{equation}

We have $\eps_0=1$, because $Y^0$ and $X^0$ are the identity matrices of size $n+1$ and $n$, respectively. Next, 
\begin{equation*}
Y^i=\begin{bmatrix} X^i & X^{i-1}x\\ 0 & 0 \end{bmatrix}, \quad i\ge1.
\end{equation*}

For a matrix $A$, let $\colsp(A)$ denote the space of vectors spanned by the columns of $A$. We have 
\begin{equation*}
\eps_i=\begin{cases} 0, & X^{i-1}x\in\colsp(X^i),\\
1, & X^{i-1}x\notin\colsp(X^i), 
\end{cases}
\qquad i=1,2,\dots\,.
\end{equation*}

On the other hand, the condition $X^{i-1}x\in\colsp(X^i)$ just means $x\in V_i$. Therefore the above formula can be rewritten as 
\begin{equation*}
\eps_i=\begin{cases} 0, & x\in V_i,\\
1, & x\notin V_i, 
\end{cases}
\qquad i=1,2,\dots\,.
\end{equation*}

Given $x\in V$, let $k=k(x)$ be the smallest positive integer such that $x\in V_k$. Then we obtain
$$
\eps_i=\begin{cases} 1, & 1\le i<k,\\
0, & i\ge k.
\end{cases}
$$

From this and \eqref{eq9.A} (and taking into account the equality $\eps_0=1$) we obtain
\begin{equation*}
\la'_i-\mu'_i=\begin{cases} 1, & i=k, \\ 0, & i\ne k. \end{cases}
\end{equation*}
In words: $\la$ is obtained from $\mu$ by adding a box to the $k$-th column. 

For such $\la$, we have, by the definition of $k=k(x)$,
$$
L(\la,\mu)=\begin{cases} 
q^{\dim V_k}- q^{\dim V_{k-1}}, & k\ge2, \\
q^{\dim V_1}, & k=1.
\end{cases}
$$ 

Finally, applying Lemma \ref{lemma9.B}, we obtain the desired expression \eqref{eq4.C}.

\subsection{One more lemma}

We keep to the definitions and notation introduced in Section \ref{sect9.1}. Assume additionally that the base field is $\FF$, the $\FF$-vector space $V$ is equipped with a nondegenerate sesquilinear Hermitian form $\tau$, and $X$ is a $\tau$-Hermitian nilpotent operator on $V$.

\begin{lemma}\label{lemma9.C}
Let $k \geq 2$. Given two vectors $x,y\in V_k$, pick a vector $x'\in V$ such that $X^{k-1}x = X^kx'$ {\rm(}which is possible by the very definition of $V_k${\rm)}. 

{\rm(i)} The quantity $\tau(X^{k-1}x',y)$ does not depend on the choice of $x'$. 

{\rm(ii)} The map $\tau_k : V_k \times V_k\to\FF$ defined by
$$
\tau_k(x,y):=\tau(X^{k-2}x,y)-\tau(X^{k-1}x',y)
$$
is a Hermitian form.

{\rm(iii)} The radical $\Rad\tau_k$ of this form is equal to $V_{k-1}$.
\end{lemma}
\begin{proof}
Consider the model of $X$ described in the proof of Proposition \ref{prop6.A}. Using it one can immediately reduce the general situation to the particular case when $X$ consists of a single Jordan block. Then we may suppose that $V$ has a basis $v_1,\dots v_p$ such that $Xv_i=v_{i-1}$ (where $v_0:=0$) and $\tau(v_i,v_j)=\de_{i+j,p+1}$. Below we write $x=(x_1,\dots,x_p)$ meaning that $x=\sum x_i v_i$. A similar notation is used for $x'$ and $y$. 

We examine separately three possible cases: 
$k\ge p+2$, $k=p+1$, and $2\le k\le p$. 

\smallskip

$\bullet$ $k\ge p+2$. Then $V_k=V_{k-1}=V$. On the other hand, $X^{k-2}x = X^{k-1}x' = 0$ for all $x, x'\in V$.
So all claims hold for trivial reasons. 

\smallskip

$\bullet$ $k=p+1$. Then $V_k=V_{p+1}=V$ and $V_{k-1}=V_p$ consists of the vectors with last coordinate $0$. Next, the condition $X^{k-1}x=X^k x'$ amounts to $X^p x=X^{p+1}x'$, which holds for any choice of $x'$ because both sides equal $0$. Thus, $x'$ may be arbitrary. On the other hand, $X^{k-1}x'=X^px'=0$, so $\tau_k(x,y)$ does not depend on $x'$ and we have
$$
\tau_k(x,y)=\tau(X^{p-1} x,y)=x_p \bar y_p.
$$
Again, all claims hold true.

\smallskip

$\bullet$ $2\le k\le p$. Then $V_k=V_{k-1}$, which means that $\tau_k$ must be identically equal to $0$. Let us take any $x, y\in V_k$ and verify that $\tau_k(x, y) = 0$. The space $V_k$ is the subspace of vectors with the last coordinates equal to $0$. Thus, $x_p=y_p=0$. Next, given $x$, we may take $x'=(0,x_1,\dots,x_{p-1})$. With this choice we have $x=Xx'$ and hence $\tau_k(x,y)=0$, as desired. Any other possible choice for $x'$ consists in adding a vector $u\in \ker X^k$.  For such a vector we have $u_{k+1}=\dots=u_p=0$. Then we have
$$
\tau(X^{k-1}u,y)=u_k \bar y_p=0,
$$
because $y_p=0$. Thus, we still get $\tau_k(x,y)=0$.   

This completes the proof.
\end{proof}

\subsection{Proof of Proposition \ref{Lmula}}

It is convenient to switch from skew-Hermitian matrices (as in the original formulation of the proposition) to Hermitian matrices. This is done simply from the relation $\Herm(N, q^2) = \eps\cdot\uu(N, q^2)$, see Section \ref{U_def}.

Let $V=\FF^N$ and $\tau: V\times V\to\FF$ be the Hermitian form given by the matrix $W=W_N$. We fix a $\tau$-Hermitian nilpotent operator $X$ on $V$, which we identify with the corresponding $N\times N$ matrix. Let  $\mu\in\Y_N$ denote the Jordan type of $X$. Next, given $x\in V$ and $z\in\F$, we form the matrix 
\begin{equation*}
Y := \begin{bmatrix} 0 & x^*W & z\\ 0 & X & x\\ 0 & 0 & 0\end{bmatrix}
\end{equation*}
and denote by $\la$ its Jordan type. Our task is  to compute the number of pairs  $(x,z)\in V\times\F$ leading to a given $\la\in\Y_{N+2}$; let us denote this quantity by $\wt L(\la,\mu)$ (instead of $\wt L^{n+1}_n(\la,\mu)$, as in the original formulation of the proposition). 

We have
\begin{equation*}
Y^i = \begin{bmatrix} 0 & x^*WX^{i-1} & z_i \\ 0 & X^i & X^{i-1}x \\  0 & 0 & 0 \end{bmatrix},\quad i\ge1,
\end{equation*}
where
\begin{equation}\label{eq9.C}
z_i=\begin{cases} z, & i=1, \\ x^* WX^{i-2}x, & i\ge2.
\end{cases}
\end{equation}
Note that
\begin{equation}\label{eq9.E}
WX^m =(X^*)^mW, \quad m\ge0,
\end{equation}
because $X$ is $\tau$-Hermitian.

As before, we set
$$
\eps_i:=\rk Y^i- \rk X^i, \quad i\ge0,
$$
and we still have
\begin{equation}\label{eq9.G}
\la'_i-\mu'_i=\eps_{i-1}-\eps_i, \quad i\ge1.
\end{equation}

In this case, $\eps_0=2$ (not $1$ as before!) and 
$$
\rk Y^i= \rk Y^{[i]}, \qquad Y^{[i]}:=
\begin{bmatrix} x^*WX^{i-1} & z_i \\ X^i & X^{i-1}x  \end{bmatrix},\quad i\ge1.
$$
Thus,
$$
\eps_i:=\rk Y^{[i]}- \rk X^i, \quad i\ge1.
$$

Recall the notation $\colsp(A)$ for the space of column vectors spanned by the columns of a given matrix $A$. Likewise, let $\rowsp(A)$ denote the space of row vectors spanned by the rows of $A$.

Observe that 
$$
X^{i-1}x\in \colsp(X^i) \Leftrightarrow x^*WX^{i-1} \in\rowsp(X^i),
$$
because $X$ is $\tau$-Hermitian. Recall that this condition is also equivalent to $x\in V_i$, see \eqref{eq9.D}. 

Given $i\ge1$, we examine two possible cases depending on whether $x$ lies in $V_i$ or not. 

\smallskip

$\bullet$ If $x\notin V_i$, then $X^{i-1}x\notin \colsp(X^i)$ and  $x^*WX^{i-1} \notin\rowsp(X^i)$, which entails that $\eps_i=2$. 

\smallskip

$\bullet$ Suppose now that $x\in V_i$. By the definition of $V_i$, there exists a vector $x'\in V$ such that $X^i x'=X^{i-1} x$. Then we may kill the lower-right block of the matrix $Y^{[i]}$ by subtracting from the last column the linear combination of the first $N$ columns with the coefficients equal to the coordinates of $x'$. Next, we can also kill the upper-left block in a similar way. This leads us to the matrix 
$$
\wt Y^{[i]}:=\begin{bmatrix} 0 & z_i-x^*W X^{i-1}x'\\ X^i & 0 \end{bmatrix}
=\begin{bmatrix} 1 & -(x')^*W\\ 0 & 1\end{bmatrix} Y^{[i]} 
\begin{bmatrix}  1 & -x'\\ 0 & 1\end{bmatrix},
$$
which has the same rank as $Y^{[i]}$ (the second equality is verified with the use of \eqref{eq9.E}). 

This gives us:
\begin{equation}\label{eq9.F}
\eps_i=\begin{cases} 0, & z_i-x^*W X^{i-1}x'=0,\\ 1, & z_i-x^*W X^{i-1}x'\ne0. \end{cases}
\end{equation}
We also note that, by virtue of \eqref{eq9.C} and Lemma \ref{lemma9.C}, 
\begin{equation*}
z_i-x^*W X^{i-1}x'=x^*W X^{i-2}x-x^*W X^{i-1}x'=\tau_i(x,x), \quad i\ge2.
\end{equation*}

\medskip

Now we are in a position to describe the possible form of $\la$ and compute the desired quantity $\wt L(\la,\mu)$. 

As before, we denote by $k=k(x)$ the least positive integer such that $x\in V_k$. Since $V_i=V$ for all $i\ge \mu_1+1$, we have $1\le k\le \mu_i+1$. We examine separately two cases linked to the alternative in \eqref{eq9.F}. Then each of them is  subdivided into two cases depending on whether $k\ge2$ or $k=1$.  

\smallskip

Case (1A): $k\ge2$ and $\tau_k(x,x)=0$. We claim that in this case
\begin{equation*}
\eps_i=\begin{cases} 2, & 0\le i\le k-1, \\0, & i\ge k.
\end{cases}
\end{equation*}
Indeed, we always have $\eps_0=2$. If $1\le i\le k-1$, then  $x\notin V_i$, whence $\eps_i=2$. For $i=k$ we have $\eps_k=0$, because $x\in V_k\setminus V_{k-1}$ and $\tau_k(x, x)=0$ by the assumption. For $i\ge k+1$ the same result holds because we still have  $x\in V_i$ and $\tau_i(x,x)=0$, where the latter equality holds for the reason that $x\in V_{i-1}$ and $V_{i-1}$ is the radical of $\tau_i$ (Lemma \ref{lemma9.C}). 

\smallskip

Case (1B): $k=1$ and $z-x^*Wx'=0$. (Here we assume that $x'$ is chosen in advance and is fixed.) We claim that in this case
\begin{equation*}
\eps_i=\begin{cases} 2, & i=0, \\0, & i\ge 1.
\end{cases}
\end{equation*}
The proof is as the same as in the previous case, the only difference is that the role of the quantity $\tau_1(x,x)$ (which is not defined) is played by $z_1-x^*Wx'=z-x^*Wx'$.

\smallskip

Case (2A): $k\ge2$ and $\tau_k(x,x)\ne0$. We claim that in this case
\begin{equation*}
\eps_i=\begin{cases} 2, & 0\le i\le k-1, \\1, & i=k, \\ 0, & i\ge k+1.
\end{cases}
\end{equation*}
Here we argue again as in Case (1A); the only difference arises for $i=k$ because now $\tau_k(x,x)\ne0$.

\smallskip

Case (2B): $k=1$ and $z-x^*Wx'\ne0$. (Again, $x'$ is fixed in advance.) We claim that in this case
\begin{equation*}
\eps_i=\begin{cases} 2, & i=0, \\1, & i=1, \\ 0, & i\ge 2.
\end{cases}
\end{equation*}
Indeed, this is checked as in Case (1B).

\smallskip

From \eqref{eq9.G} and the above expressions for $\eps_i$ it is seen that in Case (1) (that is, (1A) and (1B)), $\la$ is obtained from $\mu$ by appending a vertical domino to the $k$th column, while in Case (2) (that is, (2A) and (2B)),  one box is appended to the $k$th column and another box is appended to the $(k+1)$th column. Thus, in both cases, $\mu\nearrow\!\nearrow\la$, as claimed in Proposition \ref{Lmula}.

We proceed to the calculation of $\wt L(\la,\mu)$. We need a notation:

\begin{definition} 
Let $E$ be a finite-dimensional vector space over $\FF$ and $\kappa: E\times E\to\FF$ a nondegenerate sesquilinear Hermitian form. Then
we set 
\begin{equation}\label{eq9.H}
c_0(E):=\#\{ x\in E : \kappa(x, x) = 0 \}, \quad c_1(E):=\#\{x\in E: \kappa(x,x)\ne0\}.
\end{equation}

\medskip

Observe that the choice of $\kappa$ does not affect the definition of $c_0(E)$ and $c_1(E)$, because all forms $\kappa$ on the vector space $E$ are all equivalent.

Note that 
\begin{equation*}
c_0(E)+c_1(E)=q^{2\dim E}.
\end{equation*}
\end{definition}

We claim that the following formulas hold:

\begin{align*}
&\text{Case (1A):} \quad \wt L(\la,\mu)=q(c_0(V_k/V_{k-1})-1)q^{2\dim V_{k-1}}, \quad k=2,3,\dots, \\
&\text{Case (1B):} \quad \wt L(\la,\mu)=q^{2 \dim V_1},\\
&\text{Case (2A):} \quad \wt L(\la,\mu)=qc_1(V_k/V_{k-1})q^{2\dim V_{k-1}}, \quad k=2,3,\dots, \\
&\text{Case (2B):} \quad \wt L(\la,\mu)=q^{2 \dim V_1} (q-1).
\end{align*}

From \eqref{eq9.H} it follows that the sum of these expressions equals $q^{2\dim V+1}$, as it should be (because this is the the total number of pairs $(x,z)\in V\times \F$). 

These formulas are derived directly from the definition of the cases. For instance, let us comment on Case (1A). Here $z\in\F$ may be arbitrary, whence the factor $q$. Next, we have to count the number of vectors $x\in V_k\setminus V_{k-1}$ with $\tau_k(x,x)=0$. Recall that $V_{k-1}$ is the radical of the form $\tau_k$, so that we may treat $\tau_k$ as a nondegenerate form on $V_k/V_{k-1}$.
The number of vectors	 $x$ with the required conditions is equal to the number of pairs $(u, v)$, where $u$ is a \emph{nonzero} vector in the quotient space $V_k/V_{k-1}$ with $\tau_k(u, u) = 0$, and $v$ is any vector in $V_{k-1}$. The number of possible vectors $u$ is $c_0(V_k/V_{k-1})-1$, whereas the number of possible vectors $v$ is $q^{2\dim V_{k-1}}$ -- the cardinality of $V_{k-1}$.
The verification of the other formulas is similar. 

To complete the computation we need the following explicit expressions for $c_0(E)$ and $c_1(E)$ (see Lemma \ref{lemma9.D} below): 
\begin{equation}\label{eq9.I}
c_0(E)=q^{2m - 1} + (-1)^{m} q^{m - 1}(q - 1), \quad 
c_1(E)=q^{2m} - c_0(E),  \quad m:=\dim E. 
\end{equation}

Finally, we need explicit expressions for the dimensions entering our formulas, and they are afforded by Lemma \ref{lemma9.B}:

\begin{equation}\label{eq9.J}
\dim(V_k/V_{k-1})=m_{k-1}, \quad \dim V_{k-1}=N-\sum_{j\ge k-1}m_j, \quad \dim V_1=N - \sum_{j\ge 1} m_j.
\end{equation}

Substituting \eqref{eq9.I} and \eqref{eq9.J} into the expressions for $\wt L(\la,\mu)$ we get the formulas of Proposition  \ref{Lmula}. It remains to check \eqref{eq9.I}.

\begin{lemma}\label{lemma9.D}
The quantities $c_0(E)$ and $c_1(E)$ defined by \eqref{eq9.H} are given by \eqref{eq9.I}.
\end{lemma}

\begin{proof}
Write $c_0(m)$ and $c_1(m)$ instead of $c_0(E)$ and $c_1(E)$. The equality $c_1(m)=q^{2m}-c_0(m)$ is evident. Next, $c_0(m)$ equals the number of solutions in $\FF^m$ of the equation
$$
x_1\overline{x_1}+\dots+ x_m\overline{x_m}=0.
$$
Clearly, $c_0(1)=1$, which agrees with \eqref{eq9.I}. Next, for $m\ge2$ we set
$$
a:=x_1\overline{x_1}+\dots+ x_{m-1}\overline{x_{m-1}}
$$
and observe that the equation $x_m\overline{x_m}=-a$ has a single solution $x_m=0$ for $a=0$, and $q+1$ solutions for every $a\in\F\setminus\{0\}$. This entails the recurrence relation 
\begin{equation*}
c_0(m)=c_0(m-1)+c_1(m-1)(q+1), \quad m=2,3,\dots\,.
\end{equation*}
On  the other hand, \eqref{eq9.I} satisfies this recurrence, which proves the lemma.
\end{proof}

This completes the proof of Proposition \ref{Lmula}.

\section{Generalized spherical representations}\label{sect10}

The contents of this section are not used in the body of the paper, so we omit detailed proofs. Our purpose here is to explain a representation-theoretic meaning of coadjoint-invariant Radon measures.

Let $V$ be a locally compact Abelian (LCA) group, and $\wh V$ its Pontryagin dual (the group of characters). Next, let $\mathcal K$ be a compact group acting on $V$ by automorphisms. Then it also acts by automorphisms on $\wh V$, and we set $\mathcal G:=\mathcal K\ltimes \wh V$ (the semidirect product). By a \emph{spherical representation} of $(\mathcal G,\mathcal K)$ we mean a pair $(T,\xi)$, where $T$ is a unitary representation of the group $\mathcal G$ on a Hilbert space $H=H(T)$ and $\xi\in H$ is a distinguished cyclic $\mathcal K$-invariant vector. A well-known result due to Mackey says that there is a one-to-one correspondence $P\leftrightarrow (T,\xi)$ between finite $\mathcal K$-invariant Borel measures $P$ on $V$ and (equivalence classes of) spherical representations $(T,\xi)$ of $(\mathcal G,\mathcal K)$. 

This correspondence looks as follows. Given $P$, we form the Hilbert space $H:=L^2(V, P)$. The distinguished vector $\xi$ is the constant function equal to $1$. The representation $T$ is defined by setting for $\eta\in H$ and $v\in V$
\begin{equation}\label{eq1}
(T(k) \eta)(v) := \eta(k^{-1}\cdot v), \qquad  k\in \mathcal K,
\end{equation}
and
\begin{equation}\label{eq2}
(T(\wh v)\eta)(v) := \chi(\wh v, v), \qquad \wh v\in \wh V,
\end{equation}
where $\chi$ denotes the canonical pairing between $\wh V$ and $V$ (a bicharacter of $\wh V\times V$). Then $(T,\xi)$ is a spherical representation, and every spherical representation is obtained in this way from a (unique) measure $P$. 

We are going to extend this correspondence to the case when $\mathcal K$ is no longer compact and the measures $P$ are allowed to be infinite. On the other hand, we impose restrictions on  $V$. 

We need a few definitions. By a \emph{lattice} in a LCA group $V$ we mean an open compact subgroup; we denote by $\Lat(V)$ the set of all lattices in $V$. Next, if $X$ is a closed subgroup of a LCA group $V$, then we will denote by $X^\perp$ the subgroup of $\wh V$ formed by the characters that identically equal $1$ on $X$. Note that if $X\in\Lat(V)$, then $X^\perp\in\Lat(\wh V)$. 

We are interested in LCA groups $V$ satisfying the following condition:

(*) There exists a two-sided infinite sequence of nested lattices $\{X_n: n\in\Z\}$ such that 
$$
\cdots\supset X_n\supset X_{n+1}\supset\cdots, \qquad \bigcap X_n=\{0\}, \qquad \bigcup X_n=V.
$$
Note that if $V$ satisfies (*), then so does $\wh V$; for the corresponding chain of lattices, one can take $Y_n:=X_{-n}^\perp$. 

In what follows, we assume that $V$ satisfies (*). A basic example is the additive group of a local non-Archimedean field $\F((z))$, where $\F$ is a finite field. 

We denote by $S(V)$ the space of \emph{Schwartz--Bruhat functions} --- these are the complex valued functions on $V$ that are compactly supported and locally constant. Let also $S[\wh V]$ be the space of complex measures on $V$ which are absolutely continuous with respect to the Haar measure and whose densities are Schwartz--Bruhat functions. Both $S(V)$ and $S[\wh V]$ are commutative $*$-algebras, with respect to pointwise multiplication and convolution product, respectively. 

The Fourier transform establishes an algebra isomorphism between $S(V)$ and $S[\wh V]$. From this it is easy to obtain a version of Bochner's theorem: 

\begin{proposition}\label{prop1}
The Fourier transform establishes a bijection $P\leftrightarrow \varphi$ between Radon measures $P$ on $V$ and positive functionals $\varphi$ on the $*$-algebra $S[\wh V]$.  
\end{proposition}

Let us emphasize that it is the Radon condition that allows one to define the Fourier transform of $P$.

Suppose now that $T$ is a continuous unitary representation of the group $\wh V$ on a Hilbert space $H$. We are going to define the \emph{Gelfand triple} 
$$
H_\infty\subset H\subset H^\infty.
$$
For $X\in\Lat(\wh V)$, let $H^X$ denote the subspace of $X$-invariant vectors in $H$. By definition,
$$
H_\infty:=\bigcup_{X\in\Lat(\wh V)} H^X
$$
(the subspace of \emph{smooth vectors}). It is an inductive limit space, $H_\infty=\varinjlim H^X$, with respect to natural embeddings $H^X\to H^Y$ corresponding to pairs of lattices $X\supset Y$. We equip $H_\infty$ with the inductive limit topology, where each $H^X$ is endowed with its usual norm topology. Next, $H^\infty$ is defined as the space of continuous antilinear functionals on $H_\infty$; it is nothing else than the projective limit $\varprojlim H^X$, where, for each pair $X\supset Y$, the arrow $H^X\leftarrow H^Y$ is the natural projection. Elements of $H^\infty$ are called \emph{generalized vectors}.  

The representation $T$ gives rise to a $*$-representation of the algebra $S[\wh V]$ on $H$, which preserves the subspace $H_\infty\subset H$ and extends to the space $H^\infty$. Furthermore, if $\xi\in H^\infty$ and $f\in S[\wh V]$, then $T(f)\xi\in H_\infty$. A generalized vector $\xi$ is said to be \emph{cyclic} if the subspace $\{T(f)\xi: f\in S[\wh V]\}$ is dense in $H$.  

\begin{proposition}\label{prop2}
There exists a natural bijective correspondence $P\leftrightarrow (T,\xi)$ between Radon measures $P$ on $V$ and (equivalence classes of) pairs $(T,\xi)$, where $T$ is a unitary representation of $\wh V$ and $\xi$ is a distinguished generalized cyclic vector.
\end{proposition}

\begin{proof}[Sketch of proof]
Given $P$, the corresponding pair $(T,\xi)$ is constructed as follows. The Hilbert space $H$ is $L^2(V,P)$. The operator $T(\wh v)$ on $H$ is defined by \eqref{eq2}. The subspace $H_\infty\subset H$ consists of compactly supported $L^2$-functions. The space $H^\infty\supset H$ is formed by the locally-$L^2$-functions. The distinguished vector $\xi$ is the constant function $1$. 

Conversely, given $(T,\xi)$, the matrix element
$$
\varphi(\wh v):=(T(\wh v)\xi,\xi)
$$
is well defined,  and it is a positive functional on $S[\wh V]$. Then we use the correspondnce $\varphi\mapsto P$ from Proposition \ref{prop1}. 
\end{proof}

Suppose now that there is a group $\mathcal K$ acting on $V$ (and hence on $\wh V$) by automorphisms. Then we may form the semidirect product $\mathcal G:=\mathcal K\ltimes\wh V$.  If $T$ is a unitary representation of $\mathcal G$ on a Hilbert space $H$, then the action of $\mathcal K$ preserves the  Gelfand triple associated with the restriction of $T$ to $\wh V$. By a \emph{generalized spherical representation} of $(\mathcal G,\mathcal K)$ we mean a pair $(T,\xi)$, where $T$ is as above and $\xi\in H^\infty$ is a generalized $\mathcal K$-invariant cyclic vector $\xi$. 

Denote by $\mathcal P(V)^{\mathcal K}$ the set of $K$-invariant Radom measures on $V$. Given $P\in \mathcal P(V)^{\mathcal K}$, we assign to it a generalized spherical representation of the group $\mathcal G$ by formulas \eqref{eq1} and \eqref{eq2}; the distinguished vector $\xi$ being, as before, the constant function $1$. 

\begin{proposition}\label{prop3}
The correspondence $P\mapsto (T,\xi)$ just described establishes a bijection between the set $\mathcal P(V)^{\mathcal K}$ and the set of equivalence classes of generalized spherical representations of $(\mathcal G,\mathcal K)$. Under this correspondence, $T$ is irreducible if and only if $P$ is ergodic.
\end{proposition}

For instance, one may take (using the notation of section \ref{sect1.1}) $V=\bar\g^*$, $\wh V=\bar\g$, $\mathcal K=\overline G$. 
Then Proposition \ref{prop3} says that coadjoint-invariant Radon measures give rise to generalized spherical representations of the pair $(\overline G\ltimes\bar\g, \overline G)$ (cf. \cite[Sect. 5.1]{GKV}), and ergodic measures produce irreducible representations.

\bigskip

Cesar Cuenca:

\smallskip

$^1$Department of Mathematics, Harvard University, Cambridge, MA, USA.

\smallskip

Email address: cesar.a.cuenk@gmail.com

\bigskip

Grigori Olshanski:

\smallskip

$^2$Institute for Information Transmission Problems, Moscow, Russia;

\smallskip

$^3$Skolkovo Institute of Science and Technology, Moscow, Russia;

\smallskip

$^4$HSE University, Moscow, Russia.

\smallskip

Email address: olsh2007@gmail.com

\end{document}